\documentclass[11pt, reqno]{amsart}
\usepackage{txfonts}
\usepackage{amsmath,amssymb,amsthm,mathrsfs,enumerate,bm,xcolor,multirow,pbox}
\usepackage{graphicx,color,framed,tikz,caption,subcaption}
\usepackage{enumitem}
\setlist{leftmargin=9mm}
\usepackage[colorlinks,linkcolor=black,citecolor=black,urlcolor=black]{hyperref}
\allowdisplaybreaks[4]
\numberwithin{equation}{section}

\newcommand{\R}{\mathbb{R}}

\newcommand{\pnorm}[2]{\lVert #1\rVert_{#2}}
\newcommand{\bigpnorm}[2]{\big\lVert#1\big\rVert_{#2}}
\newcommand{\biggpnorm}[2]{\bigg\lVert#1\bigg\rVert_{#2}}
\newcommand{\abs}[1]{\lvert#1\rvert}
\newcommand{\bigabs}[1]{\big\lvert#1\big\rvert}
\newcommand{\biggabs}[1]{\bigg\lvert#1\bigg\rvert}
\newcommand{\iprod}[2]{\langle#1,#2\rangle}

\renewcommand{\epsilon}{\varepsilon}

\newcommand{\smallo}{\mathfrak{o}}
\newcommand{\bigo}{\mathcal{O}}

\newcommand{\equald}{\stackrel{d}{=}}
\renewcommand{\hat}{\widehat}
\renewcommand{\tilde}{\widetilde}

\DeclareMathOperator{\E}{\mathbb{E}}
\DeclareMathOperator{\Prob}{\mathbb{P}}

\DeclareMathOperator{\tr}{tr}
\DeclareMathOperator{\var}{Var}

\DeclareMathOperator{\op}{op}

\DeclareMathOperator{\med}{\mathsf{Med}}
\DeclareMathOperator{\lip}{Lip}

\DeclareMathOperator{\err}{\mathsf{err}}

\DeclareMathOperator{\rem}{Rem}

\theoremstyle{definition}\newtheorem{problem}{Problem}[section]
\theoremstyle{definition}\newtheorem{definition}[problem]{Definition}
\theoremstyle{remark}

\theoremstyle{remark}\newtheorem{remark}{Remark}
\theoremstyle{definition}
\theoremstyle{plain}\newtheorem{theorem}[problem]{Theorem}
\theoremstyle{plain}
\theoremstyle{plain}\newtheorem{lemma}[problem]{Lemma}
\theoremstyle{plain}\newtheorem{proposition}[problem]{Proposition}
\theoremstyle{plain}
\theoremstyle{plain}

\AtBeginDocument{%
	\def\MR#1{}
}

\begin{document}

\title[Exact bounds for quadratic empirical processes]{Exact bounds for some quadratic empirical processes with applications}
\thanks{The research of Q. Han is partially supported by NSF grants DMS-1916221 and DMS-2143468.}

\author[Q. Han]{Qiyang Han}

\address[Q. Han]{
Department of Statistics, Rutgers University, Piscataway, NJ 08854, USA.
}
\email{qh85@stat.rutgers.edu}

\date{\today}


\begin{abstract}
Let $Z_1,\ldots,Z_n$ be i.i.d. isotropic random vectors in $\R^p$, and $T \subset \R^p$ be a compact set. A classical line of empirical process theory characterizes the size of the suprema of the quadratic process 
\begin{align*}
\sup_{t \in T} \biggabs{ \frac{1}{n}\sum_{i=1}^n \iprod{Z_i}{t}^2-\pnorm{t}{}^2 },
\end{align*}
via a single parameter known as the Gaussian width of $T$.

This paper introduces an improved bound for the suprema of this quadratic process for standard Gaussian vectors $\{Z_i\}$ that can be exactly attained for certain choices of $T$, and is thus referred to as an exact bound. Our exact bound is expressed via a collection of (stochastic) Gaussian widths over spherical sections of $T$ that serves as a natural multi-scale analogue to the Gaussian width of $T$. Compared to the classical bounds for the quadratic process, our new bounds not only determine the optimal constants in the classical bounds that can be attained for some $T$, but also precisely capture certain subtle phase transitional behavior of the quadratic process beyond the reach of the classical bounds. 

To illustrate the utility of our results, we obtain tight versions of the Gaussian Dvoretzky–Milman theorem for random projection, and the Koltchinskii-Lounici theorem for covariance estimation, both with optimal constants. Moreover, our bounds recover the celebrated BBP phase transitional behavior of the top eigenvalue of the sample covariance and its generalization to the sample covariance error.

The proof of our results exploits recently sharpened Gaussian comparison inequalities. The technical scope of our method of proof is further demonstrated in obtaining an exact bound for a two-sided Chevet inequality. 
\end{abstract}

\maketitle

\vspace{-1em}
\setcounter{tocdepth}{1}
\tableofcontents

\sloppy

\section{Main results}\label{section:main_result}

\subsection{Exact bounds for the suprema of quadratic empirical processes}
Let $Z_1,\ldots,Z_n$ be i.i.d. standard Gaussian vectors in $\R^p$, and $T \subset \R^p$ be a compact set. We will be primarily interested in the behavior of the suprema of the following quadratic process
\begin{align}\label{def:expected_suprema_sp}
 \sup_{t \in T} \biggabs{ \frac{1}{n}\sum_{i=1}^n \iprod{Z_i}{t}^2-\pnorm{t}{}^2 }.
\end{align}
Before formally describing the behavior of (\ref{def:expected_suprema_sp}), with $h\sim \mathcal{N}(0,I_p)$, let
\begin{align}\label{def:rad_gw_sd_T}
r(T)\equiv \sup_{t \in T}\,\pnorm{t}{},\quad w(T)\equiv \E^{1/2} \Big(\sup_{t \in T}\,\iprod{h}{t}\Big)^{2},\quad d(T)\equiv \bigg(\frac{w(T)}{r(T)}\bigg)^2
\end{align}
be the \emph{radius}, \emph{Gaussian width} and \emph{stable dimension} of $T$ respectively; see e.g., \cite[Section 7]{vershynin2018high} for an in-depth treatment of some properties of these geometric quantities associated with $T$.

An important line of empirical process theory \cite{klartag2005empirical,adamczak2010quantitative,mendelson2007reconstruction,mendelson2010empirical,mendelson2012generic,dirksen2015tail} provides a sharp bound of the quadratic process (\ref{def:expected_suprema_sp}) using the above geometric parameters: for some universal constant $C>0$,
\begin{align}\label{eqn:chaining_sp}
\E \sup_{t \in T} \biggabs{ \frac{1}{n}\sum_{i=1}^n \iprod{Z_i}{t}^2-\pnorm{t}{}^2 }\leq C\cdot r^2(T)\cdot \bigg(\sqrt{ \frac{d(T)}{n} }+ \frac{d(T)}{n}\bigg).
\end{align}
Bounds of the type (\ref{eqn:chaining_sp}) are usually proved via chaining techniques (cf., \cite{talagrand2014upper}), which hold under more general classes of quadratic functions and distributions of $Z_i$'s. These bounds have found vast applications in signal processing and statistical applications; we only refer the readers to \cite{vershynin2018high} and references therein for a sample of these applications. 

The first goal of this paper is to provide a refined understanding for the behavior of the quadratic process (\ref{def:expected_suprema_sp}). In particular, we will provide a systematically improved bound to (\ref{eqn:chaining_sp}) which can actually be exactly attained for certain choices of $T$, and thus is refer to as an \emph{exact bound}.

The key to our improvement is to replace the `global' parameters associated with $T$ in (\ref{def:rad_gw_sd_T}) with a collection of `multiscale' parameters as follows. Let
\begin{align*}
\Lambda_T\equiv \{\pnorm{t}{}: t \in T\}.
\end{align*}
For any $\zeta \in \R$, with $h\sim \mathcal{N}(0,I_p)$, let
\begin{align}\label{def:E_pm}
E_{+,\zeta}(T) &\equiv \E\sup_{\alpha \in \Lambda_{T}}\bigg\{\bigg(\alpha+\frac{1}{\sqrt{n}}\sup_{t \in T, \pnorm{t}{}=\alpha}\iprod{h}{t}\bigg)^2-\zeta \alpha^2\bigg\},\nonumber\\
E_{-,\zeta}(T) &\equiv \E\sup_{\alpha \in \Lambda_{T}}\bigg\{\zeta\alpha^2-\bigg(\alpha-\frac{1}{\sqrt{n}}\sup_{t \in T, \pnorm{t}{}=\alpha}\iprod{h}{t}\bigg)^2_+\bigg\}. 
\end{align}
For notational convenience, we shall write 
\begin{align*}
E_\ast(T)\equiv E_{+,1}(T).
\end{align*} 
We shall also write $G \in \R^{n\times p}$ as a standard Gaussian matrix with i.i.d. $\mathcal{N}(0,1)$ entries, and $G_n\equiv G/\sqrt{n}$. Using this notation, the quadratic process in (\ref{def:expected_suprema_sp}) then can be written equivalently in law as $\sup_{t \in T}\abs{\pnorm{G_n t}{}^2-\pnorm{t}{}^2}$. 

\begin{theorem}\label{thm:square_process}
	Let  $T \subset \R^p$ be a compact set with $0 \in T$. Then there exists some universal constant $C>0$ such that for any $\zeta \in \R$ and any $x\geq 1$, with probability at least $1-e^{-x}$, 
	\begin{align*}
	\sup_{t \in T} \pm\big( \pnorm{G_n t}{}^2-\zeta \pnorm{t}{}^2 \big) \leq  E_{\pm,\zeta}(T)+C\cdot r^2(T)\cdot \bigg(\sqrt{1\vee \abs{\zeta}\vee \frac{d(T)}{n}}\sqrt{\frac{x}{n}}+\frac{x}{n}\bigg).
	\end{align*}
\end{theorem}

The condition $0 \in T$ in the above theorem is imposed for technical convenience. With this condition, the suprema on the left hand side of the above display is always non-negative.

As an immediate consequence of Theorem \ref{thm:square_process} in the main case $\zeta=1$ of our interest, we obtain the following dimension-free estimate for the quadratic process (\ref{def:expected_suprema_sp}) (without assuming $0 \in T$).
\begin{theorem}\label{thm:general_bound_sp_dT}
Let $T \subset \R^p$ be a compact set. Then with $T_0\equiv T\cup \{0\}$, there exists some universal constant $C>0$ such that for any $x\geq 1$, with probability at least $1-e^{-x}$,
\begin{align*}
&\sup_{t \in T} \bigabs{ \pnorm{G_n t}{}^2-\pnorm{t}{}^2 }\leq E_\ast(T_0)+C\cdot r^2(T)\cdot \bigg(\sqrt{1\vee \frac{d(T)}{n}}\sqrt{\frac{x}{n}}+\frac{x}{n}\bigg)\\
& \leq r^2(T)\cdot \bigg(1+\sqrt{ \frac{Cx}{d(T)} }+\frac{Cx}{\sqrt{d(T)\cdot (n\vee d(T))}}\bigg)\cdot\bigg(2\sqrt{ \frac{d(T)}{n} }+ \frac{d(T)}{n}\bigg).
\end{align*}
\end{theorem}

Clearly, by integrating the tail we immediately obtain moment estimates for the quadratic process (\ref{def:expected_suprema_sp}).

Now it is evident that Theorem \ref{thm:general_bound_sp_dT} improves the chaining bound (\ref{eqn:chaining_sp}) for the quadratic process (\ref{def:expected_suprema_sp}) in major ways: it not only leads to optimal constants in the bound (\ref{eqn:chaining_sp}) that can be attained for some $T$, as we will see later, but it also captures the correct phase transitional behavior of (\ref{def:expected_suprema_sp}) in the canonical spiked covariance model, which is fundamentally beyond the reach of the classical bound (\ref{eqn:chaining_sp}). 

We note that while the quantity $E_\ast(T)$ (or $E_\ast(T_0))$ in Theorem \ref{thm:general_bound_sp_dT} offers a systematic improvement over the classical bound (\ref{eqn:chaining_sp}), we do not expect it to be exact for every choice of $T$. For instance, for $T_\infty=[-1,1]^p$, the behavior of (much) simpler process $\sup_{v \in T_\infty} \pnorm{G_n v}{}$ (after proper normalization) is intimately related to the Parisi formula in the spin glass model \cite{chen2023gaussian} that goes beyond the Gaussian process method; see also the paragraphs following \cite[Problem II.14]{davidson2001local} for related discussion. It remains open to characterize the precise geometric properties of $T$ that make the bound via $E_\ast(T)$ exact for the quadratic process (\ref{def:expected_suprema_sp}).

\begin{remark}
Multiscale parameters similar to (\ref{def:E_pm}) have appeared in the classical Gaussian process theory for (much) simpler processes. For instance, by leveraging the classical Gaussian comparison principles, \cite[Corollary 1.2]{gordon1988milman} proved that for any compact set $T \subset \R^p$, 
\begin{align}\label{ineq:gordon}
\E \sup_{t \in T} \pnorm{G_n t}{}\leq \E \sup_{\alpha \in \Lambda_{T}} \bigg\{\alpha + \frac{1}{\sqrt{n}}\sup_{t \in T, \pnorm{t}{}=\alpha} \iprod{h}{t}\bigg\}+\frac{C\cdot  r(T)}{\sqrt{n}}. 
\end{align}
A major simplification in (\ref{ineq:gordon}) is that $\sup_{t \in T}\pnorm{G_n t}{} = \sup_{t \in T, s \in \partial B_n}\iprod{s}{G_n t}$ can be expressed as a bilinear form of $G_n$, which then allows a direct reduction to decoupled Gaussian linear forms via comparison principles. Here our Theorem \ref{thm:general_bound_sp_dT} can be viewed as a significant generalization of (\ref{ineq:gordon}) where such straightforward reduction is not possible, and therefore classical comparison inequalities cannot be applied at least directly.
\end{remark}

\subsection{An extension: exact bounds for a two-sided Chevet inequality}

The method of proof for Theorems \ref{thm:square_process} and \ref{thm:general_bound_sp_dT}, as will be clear from Section \ref{section:proof_outline} ahead, is based on a recently sharpened Gaussian comparison principle. As an illustration of the technical flexibility of this proof method, we will also provide below an improved, exact bound for 
\begin{align}\label{def:chevet_proc}
\sup_{t \in T} \Big|\sup_{s \in S}\iprod{s}{G_n t}- \zeta\pnorm{t}{}\Big|,\quad \zeta \in \R,
\end{align}
where $T\subset \R^p$ and $S \subset \R^n$ are two compact sets. A canonical choice of $\zeta$ in the above display (\ref{def:chevet_proc}) is 
\begin{align}\label{def:zetaS}
\zeta_S\equiv  \frac{w_1(S)}{\sqrt{n}}\equiv \frac{1}{\sqrt{n}}\E \sup_{s \in S}\iprod{g}{s},\quad g \sim \mathcal{N}(0,I_n).
\end{align}
For this choice of $\zeta=\zeta_S$, it is well known that there exists some universal constant $C>0$ such that
\begin{align}\label{ineq:chevet_twosided}
\E \sup_{t \in T} \Big|\sup_{s \in S}\iprod{s}{G_n t}- \zeta_S\pnorm{t}{}\Big|\leq \frac{C\cdot w(T)r(S)}{\sqrt{n}}. 
\end{align}
The readers are referred to \cite[Chapters 9 and 11]{vershynin2018high} and a detailed account and numerous applications of the two-sided Chevet inequality (\ref{ineq:chevet_twosided}) (some results therein are proved for special choices of $S$, say, $S=B_n$ being the unit ball).

To describe our improved bound for (\ref{def:chevet_proc}), for any $\zeta \in \R$, with $g\sim \mathcal{N}(0,I_n)$ and $h\sim \mathcal{N}(0,I_p)$ being independent standard Gaussian vectors, let
\begin{align}\label{def:E_pm_sqrt}
E_{+,\zeta}^\circ(T,S)&\equiv \E \sup_{\alpha \in \Lambda_T,\beta \in \Lambda_S}\bigg\{ \frac{\alpha}{\sqrt{n}}\sup_{s \in S, \pnorm{s}{}=\beta}\iprod{g}{s}+ \frac{\beta}{\sqrt{n}} \sup_{t\in T,\pnorm{t}{}=\alpha} \iprod{h}{t}-\zeta\alpha\bigg\},\nonumber\\
E_{-,\zeta}^\circ(T,S)&\equiv \E \sup_{\alpha \in \Lambda_T}\inf_{\beta \in \Lambda_S}\bigg\{\zeta\alpha- \frac{\alpha}{\sqrt{n}}\sup_{s \in S, \pnorm{s}{}=\beta}\iprod{g}{s}+ \frac{\beta}{\sqrt{n}} \sup_{t\in T,\pnorm{t}{}=\alpha} \iprod{h}{t}\bigg\}.
\end{align}

\begin{theorem}\label{thm:sqrt_process}
	Let $T\subset \R^p$ and $S \subset \R^n$ be two compact sets with $0 \in T$. Then there exists some universal constant $C>0$ such that for any $\zeta \in \R$ and any $x\geq 1$, with probability at least $1-e^{-x}$, 
	\begin{align*}
	\sup_{t \in T} \pm\bigg(\sup_{s \in S}\iprod{s}{G_n t}- \zeta\pnorm{t}{}\bigg)\leq E_{\pm,\zeta}^\circ(T,S)+ C\cdot r(T)r(S) \sqrt{\frac{x}{n}}.
	\end{align*}
\end{theorem}

Again here $0 \in T$ in the above theorem is assumed for technical convenience; under this condition the left hand side of the above display is always non-negative.

As a fairly easy consequence of Theorem \ref{thm:sqrt_process} above, we obtain an exact bound for the process (\ref{def:chevet_proc}).

\begin{theorem}\label{thm:sqrt_process_zetaS}
	Let $T\subset \R^p$ and $S \subset \R^n$ be two compact sets. Then with $T_0\equiv T\cup \{0\}$, there exists some universal constant $C>0$ such that for any $x\geq 1$, with probability at least $1-e^{-x}$, 
	\begin{align*}
	&\sup_{t \in T} \Big|\sup_{s \in S}\iprod{s}{G_n t}- \zeta_S\pnorm{t}{}\Big|\\
	&\leq \max_{*=\pm}E_{\ast,\zeta_S}^\circ(T_0,S)+ C\cdot r(T)r(S) \sqrt{\frac{x}{n}} \leq \frac{w(T)r(S)}{\sqrt{n}}\cdot \bigg(1+\sqrt{ \frac{C x}{d(T)} }\bigg).
	\end{align*}
\end{theorem}

Compared to (\ref{ineq:chevet_twosided}), it is easy to see that the above theorem provides a sharp leading constant $1$ for the two-sided Chevet inequality.

\begin{remark}
For the case $\zeta=0$, Theorem \ref{thm:sqrt_process} can be easily obtained via standard Gaussian comparison and concentration techniques; see, e.g., \cite[Exercise 8.7.4]{vershynin2018high}.
\end{remark}

\subsection{Organization}

The rest of the paper is organized as follows. In Section \ref{section:application} we shall provide applications of the main theorems above to random projection and covariance estimation. An outline of the proof for Theorem \ref{thm:square_process} and some further notation/definitions are given in Section \ref{section:proof_outline}. All proof details are then presented in Sections \ref{section:proof_main_all}-\ref{section:proof_application}. For the sake of completeness, we include the proof of our version of the Gaussian min-max Theorem \ref{thm:CGMT} in Appendix \ref{section:proof_CGMT}.

\subsection{Notation}\label{section:notation}

For any pair of positive integers $m\leq n$, let $[m:n]\equiv \{m,\ldots,n\}$, and $(m:n]\equiv [m:n]\setminus\{m\}$, $[m:n)\equiv [m:n]\setminus\{n\}$. We often write $[n]$ for $[1:n]$. For $a,b \in \R$, $a\vee b\equiv \max\{a,b\}$ and $a\wedge b\equiv\min\{a,b\}$. For $a \in \R$, let $a_+\equiv a\vee 0$ and $a_- \equiv (-a)\vee 0$. 

For $x \in \R^n$, let $\pnorm{x}{q}=\pnorm{x}{\ell_q(\R^n)}$ denote its $q$-norm $(0\leq q\leq \infty)$ with $\pnorm{x}{2}$ abbreviated as $\pnorm{x}{}$. Let $B_n(r;x)\equiv \{z \in \R^n: \pnorm{z-x}{}\leq r\}$, and recall that $B_n$ is the unit $\ell_2$ ball in $\R^n$. We write $\partial B_n\equiv \{x \in \R^n: \pnorm{x}{}=1\}$ as the sphere in $\R^n$. For $x,y \in \R^n$, we write $\iprod{x}{y}\equiv \sum_{i=1}^n x_iy_i$. For a matrix $M$, let $\pnorm{M}{F}$ and $\pnorm{M}{\op}$  be its Frobenius norm and spectral/operator norm respectively. For any real symmetric matrix $A\in \R^{p\times p}$, let $\lambda_\pm\big(A\big) \equiv \sup_{v \in B_p} \iprod{v}{(\pm A)v}$. Then $\pnorm{A}{\op}=\lambda_+(A)\vee \lambda_-(A)$.

We use $C_{x}$ to denote a generic constant that depends only on $x$, whose numeric value may change from line to line unless otherwise specified. $a\lesssim_{x} b$ and $a\gtrsim_x b$ mean $a\leq C_x b$ and $a\geq C_x b$ respectively, and $a\asymp_x b$ means $a\lesssim_{x} b$ and $a\gtrsim_x b$. $\bigo$ and $\smallo$ denote the usual big and small O notation. 

\section{Applications}\label{section:application}

\subsection{Application I: a tight Gaussian Dvoretzky–Milman Theorem}

We first need the following definition.

\begin{definition}
For a compact set $T\subset \R^p$, let
\begin{align*}
r_-(T)\equiv \sup \{r\geq 0: B_p(r)\subset T\},\, r_+(T)\equiv \inf \{r\geq 0: T\subset B_p(r)\}
\end{align*}
be the maximum/minimum radius of a centered ball contained in/containing $T$.
\end{definition}

A classical result of Dvoretzky and Milman \cite{dvoretzky1959theorem,dvoretzky1961some,milman1971new} for Gaussian random projection says the following: for any compact set $T \in \R^p$, the convex hull of the random Gaussian projection $G_n T \subset \R^n$ is approximately a round ball of radius $n^{-1/2}w_1(T)$, whenever the projection dimension $n\ll d(T)$. More precisely, \cite[Theorem 11.3.1]{vershynin2018high} shows that for some universal constant $C>0$, with high probability (say, at least 0.99), 
\begin{align}\label{ineq:DM_r_classical}
\bigg\{\frac{w_1(T)}{\sqrt{n}}-C r(T)\bigg\}_+  &\leq r_-\big(\mathrm{conv}(G_n T)\big)\leq r_+\big(\mathrm{conv}(G_n T)\big)\leq \frac{w_1(T)}{\sqrt{n}}+C r(T).
\end{align}
For a full literature review of this classical result, the readers are referred to \cite[Chapter 5 and Section 9.2]{artstein2015asymptotic}, \cite[Section 9.1]{ledoux2013probability}, \cite[Section 11.4]{vershynin2018high} and references therein; see also \cite[Theorem 1.6]{han2022gaussian} for a local version of this result.

The following theorem provides a tight version of the Gaussian Dvoretzky–Milman theorem. 

\begin{theorem}\label{thm:DM}
Let $T \subset \R^p$ be a compact set with $0 \in T$.  With $h\sim \mathcal{N}(0,I_p)$, let
\begin{align}\label{def:DM_r}
\mathfrak{r}_{T,+}&\equiv \frac{w_1(T)}{\sqrt{n}}+\bigg\{\E \sup_{\alpha \in \Lambda_{T}}\bigg( \frac{1}{\sqrt{n}}\sup_{t \in T, \pnorm{t}{}=\alpha}\iprod{h}{t}+\alpha\bigg)-\frac{w_1(T)}{\sqrt{n}}  \bigg\}_+,\nonumber\\
\mathfrak{r}_{T,-}&\equiv \frac{w_1(T)}{\sqrt{n}}-\bigg\{ \frac{w_1(T)}{\sqrt{n}} - \E\sup_{\alpha \in \Lambda_{T}}\bigg( \frac{1}{\sqrt{n}}\sup_{t \in T, \pnorm{t}{}=\alpha}\iprod{h}{t}-\alpha\bigg) \bigg\}_+.
\end{align}
Then there exists some universal constant $C>0$ such that for any $x\geq 1$, with probability at least $1-e^{-x}$,
\begin{align*}
\mathfrak{r}_{T,-}-C\cdot r(T) \sqrt{ \frac{x}{n} } &\leq r_-\big(\mathrm{conv}(G_n T)\big)\leq r_+\big(\mathrm{conv}(G_n T)\big)\leq \mathfrak{r}_{T,+}+C\cdot r(T) \sqrt{ \frac{x}{n} }.
\end{align*}
\end{theorem}
The proof of the above theorem can be found in Section \ref{subsection:proof_DM}. From (\ref{def:DM_r}), it is easy to see that
\begin{align}\label{ineq:DM_r}
\bigg\{\frac{w_1(T)}{\sqrt{n}}-r(T)\bigg\}_+ \leq \mathfrak{r}_{T,-}\leq \mathfrak{r}_{T,+}\leq \frac{w_1(T)}{\sqrt{n}}+r(T).
\end{align}
Then Theorem \ref{thm:DM} implies that with very high probability, $\mathrm{conv}(G_n T)$ is contained in a centered ball of radius $n^{-1/2}w_1(T)+ r(T)$ and contains a centered ball of radius $\big(n^{-1/2}w_1(T)- r(T)\big)_+$ (ignoring fluctuation terms of a smaller order). It is easy to see that the constant $1$ before $r(T)$ is optimal. In fact, consider the regime $p\asymp n$ and the example $T=B_p$. Then $r(T)=1$ and $w_1(T)=\sqrt{p}\cdot(1+\mathfrak{o}(1))$. The characterization condition (\ref{ineq:DM_variational}) in the proof ahead shows that with high probability,
\begin{align*}
r_-\big(\mathrm{conv}(G_n T)\big) = \big(\sqrt{{p}/{n}}-1+\mathfrak{o}(1)\big)_+,\quad r_+\big(\mathrm{conv}(G_n T)\big) = \sqrt{{p}/{n}}+1+\mathfrak{o}(1),
\end{align*}
matching exactly the upper and lower bounds obtained in (\ref{ineq:DM_r}).

\subsection{Application II: covariance estimation}\label{subsection:application_cov_est}
Let $X_1,\ldots,X_n \in \R^p$ be i.i.d. centered Gaussian random vectors with covariance $\Sigma\in \R^{p\times p}$, and let $
\hat{\Sigma}\equiv n^{-1}\sum_{i=1}^n X_iX_i^\top$ be the sample covariance.

\subsubsection{A dimension-free Koltchinskii-Lounici theorem with optimal constants}
A significant line of non-asymptotic theory of the sample covariance $\hat{\Sigma}$ concerns its estimation error for the underlying covariance $\Sigma$ measured in the spectral (operator) norm, i.e., $\pnorm{\hat{\Sigma}-\Sigma}{\op}$. We only refer the readers to the more recent references, e.g., \cite{tropp2016expected,liaw2017simple,koltchinskii2017concentration,minsker2017some,van2017structured,vershynin2018high,tikhomirov2018sample,brailovskaya2022universality,bandeira2023matrix,bandeira2024matrix,zhivotovskiy2024dimension}; many more earlier references can be found therein. 

An important milestone of this program is achieved in the path-breaking work of Koltchinskii and Lounici \cite{koltchinskii2017concentration} which shows that
\begin{align}\label{eqn:intro_KL}
\E \big\{\pnorm{\hat{\Sigma}-\Sigma}{\op}/ \pnorm{\Sigma}{\op} \big\} \asymp \sqrt{ \frac{\mathsf{r}(\Sigma)}{n}  }+ \frac{\mathsf{r}(\Sigma)}{n},
\end{align}
where the `effective rank' $\mathsf{r}(\Sigma)$ is defined by $\mathsf{r}(\Sigma)\equiv \tr(\Sigma)/\pnorm{\Sigma}{\op}$.

The original proof for the upper bound of (\ref{eqn:intro_KL}) in \cite{koltchinskii2017concentration} is based on a chaining estimate of \cite{mendelson2010empirical} that actually holds beyond Gaussian data. For Gaussian $\{X_i\}$'s, the upper bound of (\ref{eqn:intro_KL}) can also be proved via Slepian-Fernique comparison inequalities, cf. \cite{van2017structured}. Here we will apply Theorem \ref{thm:general_bound_sp_dT} to obtain a version of (\ref{eqn:intro_KL}) with tight constants. 
\begin{theorem}\label{thm:improved_KL}
	There exists some universal constant $C>0$ such that
	\begin{align}\label{eqn:KL_exact_3}
	\E \big\{\pnorm{\hat{\Sigma}-\Sigma}{\op}/\pnorm{\Sigma}{\op}\big\}&\leq E_\ast\big(T_\Sigma/r(T_\Sigma)\big)+ \frac{C }{\sqrt{n}}\sqrt{1\vee \frac{\mathsf{r}(\Sigma)}{n}} \nonumber \\
	&\leq \bigg(1+\frac{C}{\sqrt{\mathsf{r}(\Sigma)}}\bigg)\cdot \bigg(2\sqrt{\frac{\mathsf{r}(\Sigma)}{n}}+ \frac{\mathsf{r}(\Sigma)}{n}\bigg).
	\end{align}
	Here $T_\Sigma\equiv \Sigma^{1/2}(B_p)$.
\end{theorem}
\begin{proof}
An easy calculation shows that
\begin{align*}
\E \pnorm{\hat{\Sigma}-\Sigma}{\op} 
&= \E \sup_{t \in \Sigma^{1/2}(B_p)}\bigabs{ \pnorm{G_n t}{}^2-\pnorm{t}{}^2 }.
\end{align*}
Consequently we only need to compute the radius, Gaussian width and stable dimension of $T_\Sigma= \Sigma^{1/2}(B_p)$. Some calculations show that 
\begin{align*}
r(T_\Sigma)=\pnorm{\Sigma}{\op}^{1/2},\quad w(T_\Sigma)=\big\{\tr(\Sigma)\big\}^{1/2},\quad d(T_\Sigma)=\mathsf{r}(\Sigma).
\end{align*}
Now we may apply Theorem \ref{thm:general_bound_sp_dT} to conclude. 
\end{proof}
Clearly, the constants $2$ before $\sqrt{\mathsf{r}(\Sigma)/n}$ and $1$ before $\mathsf{r}(\Sigma)/n$ in (\ref{eqn:KL_exact_3}) are optimal in view of the Bai-Yin law \cite{bai1993limit} for the isotropic case $\Sigma=I_p$. Nor can the multiplicative factor $1+C/\sqrt{\mathsf{r}(\Sigma)}$ be removed by considering, e.g., the simplest one-dimensional case where $X_1,\ldots,X_n$'s are i.i.d. $\mathcal{N}(0,1)$ with $\mathsf{r}(\Sigma)=1$.

\begin{remark}
	A weaker bound in a similar spirit to (\ref{eqn:KL_exact_3}) is obtained in \cite[Theorem 2.1]{cai2022nonasymptotic} with additional logarithmic dependence on the dimension; see also the discussion after \cite[Proposition 3.12]{bandeira2023matrix} for a dimension-dependent version of (\ref{eqn:KL_exact_3}). Those dimension-dependent logarithm terms cannot be removed as the bounds therein cover general non-homogeneous settings. 
\end{remark}

\begin{remark}
As will be clear from below, the quantity $ E_\ast(T_\Sigma)$ precisely capture certain phase transitional behavior for $\pnorm{\hat{\Sigma}-\Sigma}{\op}$ in the spiked model (\ref{def:spiked_cov}) below, beyond the reach of any characterization via the effective rank $\mathsf{r}(\Sigma)$ alone. A matching lower bound was claimed in a preliminary arXiv version of this paper, but the proof is flawed. A rigorous proof for the lower bound is given in the recent work of \cite{bandeira2024matrix} in the spiked model via a suite of different techniques. A purely Gaussian process proof for this lower bound with possible dimension-free error terms remains open. 
\end{remark}

\subsubsection{Recovering BBP phase transitions and beyond}

Let us illustrate now the tightness of Theorem \ref{thm:improved_KL} in the popular class of spiked covariance models \cite{johnstone2001distribution}: For a fixed orthonormal system $\{\mathsf{v}_j\}_{j \in [p]}$ of $\R^p$, we consider
\begin{align}\label{def:spiked_cov}
\Sigma_{r,\lambda}\equiv I_p+ \lambda \sum_{j=1}^r \mathsf{v}_j\mathsf{v}_j^\top,\quad (r,\lambda)\in [p]\times \R_{\geq 0}.
\end{align}
To this end, for any $\delta\geq 0$, let $\overline{\Psi}_\delta, \Psi_\delta: \R_{\geq 0}\to \R_{\geq 0}$ be defined by
\begin{align}\label{def:psi_spiked_cov}
\overline{\Psi}_\delta(\lambda)&\equiv \bigg(1+\frac{\delta}{\lambda}\bigg)\big(1+\lambda\big),\nonumber\\
\Psi_\delta(\lambda)&\equiv \frac{\lambda+1}{\sqrt{\delta+4\lambda}}\bigg[2\sqrt{\delta}+\bigg(\frac{\sqrt{\delta}+\sqrt{\delta+4\lambda}}{2\lambda}\bigg)\cdot \delta\bigg].
\end{align}
With these notations, we have the following.

\begin{proposition}\label{prop:spiked_cov}
	There exists some universal constant $C>0$ such that for any $(r,\lambda) \in [p]\times \R_{\geq 0}$, with $\delta\equiv (p-r)/n$,
	\begin{align*}
	&\bigabs{E_{+,0}(T_\Sigma)-\overline{\Psi}_\delta\big(\lambda \vee \sqrt{\delta}\big)}\vee \bigabs{E_{+,1 }(T_\Sigma)-\Psi_\delta\big(\lambda \vee \{1+\sqrt{\delta}\}\big)}\\
	&\leq C(1+\lambda)\cdot\bigg\{ \sqrt{1+\frac{\delta}{1+\lambda}+\frac{r}{n}} \cdot\sqrt{\frac{r}{n}}\bigg\}.
	\end{align*}
\end{proposition}

The proof of the above proposition involves some book-keeping calculations so will be deferred to Section \ref{subsection:proof_spiked_all}.

Clearly, the quantity $E_{+,0}(T_\Sigma)$ recovers the celebrated BBP phase transition \cite{baik2005phase} for $\pnorm{\hat{\Sigma}}{\op}$ that occurs at the signal strength $\lambda=\sqrt{\delta}$. Interestingly, the calculation for $E_{+,1 }(T_\Sigma)$ suggests a phase transition at a different signal strength $\lambda=1+\sqrt{\delta}$ for the sample covariance error $\pnorm{\hat{\Sigma}-\Sigma}{\op}$ under the spiked model  (\ref{def:spiked_cov}). A rigorous proof for the lower bound is recently obtained in \cite{bandeira2024matrix}.

\section{Proof outline}\label{section:proof_outline}

\subsection{Technical tools}
The main technical tool we will be using in proving our main exact bounds is  the following version of the Gaussian min-max theorem.

\begin{theorem}[Gaussian min-max Theorem]\label{thm:CGMT}
	Suppose $D_u \subset \R^{n_1+n_2}, D_v \subset \R^{m_1+m_2}$ are compact sets, and $Q: D_u\times D_v \to \R$ is continuous. Let $G=(G_{ij})_{i \in [n_1],j\in[m_1]}$ with $G_{ij}$'s i.i.d. $\mathcal{N}(0,1)$, and $g \sim \mathcal{N}(0,I_{n_1})$, $h \sim \mathcal{N}(0,I_{m_1})$ be independent Gaussian vectors. For $u \in \R^{n_1+n_2}, v \in \R^{m_1+m_2}$, write $u_1\equiv u_{[n_1]}\in \R^{n_1}, v_1\equiv v_{[m_1]} \in \R^{m_1}$. Define
	\begin{align}
	\Phi^{\textrm{p}}_- (G)& = \max_{u \in D_u}\min_{v \in D_v} \Big( u_1^\top G v_1 + Q(u,v)\Big),\nonumber\\
	\Phi^{\textrm{p}}_+ (G)& = \min_{v \in D_v}\max_{u \in D_u} \Big( u_1^\top G v_1 + Q(u,v)\Big)\nonumber\\
	\Phi^{\textrm{a}}(g,h)& = \max_{u \in D_u}\min_{v \in D_v} \Big(\pnorm{v_1}{} g^\top u_1 + \pnorm{u_1}{} h^\top v_1+ Q(u,v)\Big).
	\end{align}
	Then for all $t \in \R$,
	\begin{align*}
	\Prob\big(\Phi^{\textrm{p}}_- (G)\geq t\big)&\leq 2 \Prob\big(\Phi^{\textrm{a}}(g,h)\geq t\big),\\
	\Prob\big(\Phi^{\textrm{p}}_+ (G)\leq t\big)&\leq 2 \Prob\big(\Phi^{\textrm{a}}(g,h)\leq t\big).
	\end{align*}	
\end{theorem}

The above version of the Gaussian min-max theorem is essentially contained in the proof of \cite[Corollary G.1]{miolane2021distribution} via Gordon's min-max theorem \cite{gordon1985some,gordon1988milman}; for completeness we spell out the details in the Appendix \ref{section:proof_CGMT}. 

A major technical merit of this recent version of Gaussian min-max theorem, as opposed to its classical version in, e.g., \cite{gordon1985some,gordon1988milman,ledoux2013probability,artstein2015asymptotic}, is that the two random optimization problems $\Phi^{\textrm{p}}_\pm (G)$ and $\Phi^{\textrm{a}}(g,h)$ are tightly related via \emph{tail bounds}. Moreover, under strong duality with $\Phi^{\textrm{p}}_+ (G)=\Phi^{\textrm{p}}_- (G)$, we may completely determine its value through $\Phi^{\textrm{a}}(g,h)$. In this sense this version of the Gaussian comparison principle is expected to be tight in a wide range of random optimization problems. The reader is referred to \cite{stojnic2013framework,thrampoulidis2018precise,miolane2021distribution,han2023noisy,han2023universality,han2023distribution,montanari2023generalization,celentano2023lasso} for a partial list of recent works on applications of this two-sided comparison principle to a number of concrete statistical problems. 

The application of the Gaussian min-max theorem is often coupled with the following classical Sion's min-max theorem. 

\begin{lemma}[Sion's min-max theorem]\label{lem:sion_minmax}
	Let $X$ be a compact convex subset of a linear topological space and $Y$ a convex subset of a linear topological space. If $f$ is a real-valued function on $X\times Y$ satisfying:
	\begin{enumerate}
		\item $y\mapsto f(x,y)$ is upper-semicontinuous and quasi-concave for all $x \in X$;
		\item$x\mapsto f(x,y)$ is lower-semicontinuous and quasi-convex  for all $y \in Y$.
	\end{enumerate}
	Then $\min_{x \in X} \sup_{y \in Y} f(x,y)= \sup_{y \in Y} \min_{x \in X}f(x,y)$.
\end{lemma}

\subsection{Further notation and definitions}

Recall that $G \in \R^{n\times p}$ has i.i.d. $\mathcal{N}(0,1)$ entries and $G_n= G/\sqrt{n}$. Let $g,h$ be independent standard Gaussian vectors in $\R^n$ and $\R^p$ respectively.  For simplicity, we shall assume that $G,g,h$ are defined on the same probability space and are mutually independent. Let
\begin{align}
\mathscr{E}_\pm\equiv \sup_{v \in T} \Big(\pm \pnorm{G_n v}{}^2 \mp \zeta\pnorm{v}{}^2\Big).
\end{align}
Let the primal and Gordon cost functions associated with $\mathscr{E}_\pm$ be
\begin{align}\label{def:gordon_cost_h_l}
\mathfrak{h}_\pm(u,v,w)&\equiv \pm\pnorm{w}{}^2 \mp \zeta\pnorm{v}{}^2 + \iprod{u}{G_n v-w},\nonumber\\
\mathfrak{l}_\pm(u,v,w)& \equiv n^{-1/2}\pnorm{v}{} \iprod{g}{u}+ n^{-1/2}\pnorm{u}{}\iprod{h}{v}\pm \pnorm{w}{}^2 \mp \zeta\pnorm{v}{}^2-\iprod{u}{w}.
\end{align}

\subsection{Proof outline for Theorem \ref{thm:square_process}}

We shall focus on the case $\mathscr{E}_+$ as the other case $\mathscr{E}_-$ shares a similar strategy. First we rewrite $\mathscr{E}_+$ into a max-min optimization problem:
\begin{align}\label{eqn:proof_sketch_1}
\mathscr{E}_+& = \sup_{v \in T} \big\{\pnorm{G_n v}{}^2- \zeta\pnorm{v}{}^2 \big\} = \sup_{v \in T,w \in \R^n}\inf_{u \in \R^n} \mathfrak{h}_+(u,v,w).
\end{align}
Now ideally we would like to apply the Gaussian min-max theorem  (cf. Theorem \ref{thm:CGMT}) to reduce the right hand side of (\ref{eqn:proof_sketch_1}) to the Gordon's max-min problem:
\begin{align}\label{eqn:proof_sketch_2}
\sup_{v \in T,w\in \R^n}\inf_{u \in \R^n} \mathfrak{h}_+(u,v,w) \stackrel{?}{\leq_{\mathbb{P}}} \sup_{v \in T,w\in \R^n}\inf_{u \in \R^n} \mathfrak{l}_+(u,v,w).
\end{align}
Let us assume that (\ref{eqn:proof_sketch_2}) can be justified properly, and proceed with computing the right hand side of the above display:
\begin{align}\label{eqn:proof_sketch_3}
&\sup_{v \in T,w \in \R^n}\inf_{u \in \R^n} \mathfrak{l}_+(u,v,w)\nonumber\\
&= \sup_{v \in T,w \in \R^n}\inf_{u \in \R^n}\bigg\{ \frac{1}{\sqrt{n}}\Big(\pnorm{v}{} \iprod{g}{u}+ \pnorm{u}{}\iprod{h}{v}\Big)-\iprod{u}{w} + \pnorm{w}{}^2-\zeta\pnorm{v}{}^2 \bigg\}\nonumber\\
& \stackrel{(\ast)}{=} \sup_{\substack{v \in T, w \in \R^n}}\inf_{\beta \geq 0} \bigg\{\frac{\beta}{\sqrt{n}} \iprod{h}{v}-\beta\cdot \biggpnorm{ \frac{1}{\sqrt{n}} \pnorm{v}{} g-w }{}+\pnorm{w}{}^2- \zeta\pnorm{v}{}^2   \bigg\}\nonumber\\
& \stackrel{(\ast\ast)}{=} \sup_{\substack{\alpha \in \Lambda_T, \\v \in T, \pnorm{v}{}=\alpha,  w \in \R^n}}\inf_{\beta \geq 0} \bigg\{\beta\Big(\langle h, v\rangle-\pnorm{ \alpha g-\sqrt{n}w }{}\Big)+ \pnorm{w}{}^2- \zeta\alpha^2  \bigg\}\nonumber\\
& = \sup_{\alpha \in \Lambda_T, w \in \R^n} \big\{ \pnorm{w}{}^2-\zeta\alpha^2 \big\}\quad \hbox{s.t. } \sup_{v \in T, \pnorm{v}{}=\alpha }\iprod{h}{v}\geq \pnorm{\alpha g-\sqrt{n}w}{}.
\end{align}
Here $(\ast)$ follows by taking the infimum over $u$ in the order $\inf_{\beta \geq 0}\inf_{\pnorm{u}{}=\beta}$, and $(\ast\ast)$ follows by rescaling $\beta$ and taking supremum over $v$ in the order $\sup_{\alpha \in \Lambda_T}\sup_{v \in T, \pnorm{v}{}=\alpha}$. 

From here, by writing $w= \alpha(g/\sqrt{n})+\tau r$ with $ \tau \in \big[0, n^{-1/2}\sup_{v \in T, \pnorm{v}{}=\alpha }\iprod{h}{v}\big]$ and $r \in \partial B_n$, the value of the optimization problem on the right hand side of (\ref{eqn:proof_sketch_3}) is the same as 
\begin{align*}
&\sup_{\alpha \in \Lambda_T, r \in \partial B_n, \tau} \Big\{ \bigpnorm{\alpha (g/\sqrt{n})+\tau r}{}^2-\zeta\alpha^2\Big\}\\
& = \sup_{\alpha \in \Lambda_T,  \sup_{ v \in T, \pnorm{v}{}=\alpha }\langle h, v\rangle\geq 0 } \bigg\{\bigg(\alpha\cdot \frac{\pnorm{g}{}}{\sqrt{n}}+\frac{1}{\sqrt{n}}\sup_{ v \in T, \pnorm{v}{}=\alpha }\langle h, v\rangle\bigg)^2-\zeta \alpha^2\bigg\} \\
& \leq  \sup_{\alpha \in \Lambda_T} \bigg\{\bigg(\alpha\cdot \frac{\pnorm{g}{}}{\sqrt{n}}+\frac{1}{\sqrt{n}}\sup_{ v \in T, \pnorm{v}{}=\alpha }\langle h, v\rangle\bigg)^2-\zeta \alpha^2\bigg\} \stackrel{(\triangle)}{\approx } E_{+,\zeta}(T),
\end{align*}
where in $(\triangle)$ we expect to use suitable concentration.

The major technical work in Section \ref{section:proof_main_all} below rigorously justifies the above informal calculations. In view of the Gaussian min-max Theorem \ref{thm:CGMT}, this unavoidably involves intricate compactification arguments with manageable error terms in both (\ref{eqn:proof_sketch_2}) and (\ref{eqn:proof_sketch_3}), and a sharp concentration inequality for $\mathscr{E}_\pm$. Using a similar technique, we will prove Theorem \ref{thm:sqrt_process} in Section \ref{section:proof_sqrt_process}.

\section{Proofs of Theorems \ref{thm:square_process} and \ref{thm:general_bound_sp_dT}}\label{section:proof_main_all}

\subsection{Some further notation}
Let
\begin{align}\label{def:err}
\err_n(T)\equiv \frac{r^2(T)}{\sqrt{n}}\sqrt{1\vee \abs{\zeta}\vee \frac{d(T)}{n}}.
\end{align}
We define
\begin{align}\label{def:F}
\mathsf{F}_{+}(h)&\equiv \sup_{\alpha \in \Lambda_T} \bigg\{\bigg(\alpha+ n^{-1/2}\sup_{v \in T, \pnorm{v}{}=\alpha} \iprod{h}{v} \bigg)^2-\zeta\alpha^2\bigg\},\nonumber\\
\mathsf{F}_{-}(h)&\equiv \sup_{\alpha \in \Lambda_T} \bigg\{\zeta\alpha^2-\bigg(\alpha- n^{-1/2}\sup_{v \in T, \pnorm{v}{}=\alpha} \iprod{h}{v} \bigg)_+^2\bigg\}.
\end{align}
Throughout this section, let for $x\geq 1$
\begin{align}\label{def:L_Delta}
L_w(x)&\equiv K_w\cdot r(T)\bigg(1+\sqrt{\abs{\zeta}}+\sqrt{ \frac{d(T)+x}{n} }\bigg),\nonumber\\
 \Delta(x)&\equiv K_w\cdot r^2(T)\cdot\sqrt{\frac{x}{n}}\cdot \bigg(1+\sqrt{\frac{d(T)+x}{n}}\bigg),
\end{align}
where $K_w>0$ is a large enough absolute constant.

\subsection{Reduction via Gaussian min-max theorem}\label{section:reduction_cgmt}
Consider the `high probability' event
\begin{align}\label{def:event_E}
E(x)&\equiv \bigg\{ \sup_{v \in T}\pnorm{G v}{}\leq C\cdot r(T)\big(\sqrt{n}+\sqrt{ d(T) }+\sqrt{x} \bigg)\bigg\} \cap \bigg\{ \bigabs{\pnorm{g}{}-\sqrt{n}}\leq C\sqrt{x} \bigg\}\nonumber\\
&\qquad \cap \bigg\{\sup_{v \in T}\abs{\iprod{h}{v}}\leq C\cdot r(T)\big(\sqrt{d(T)}+\sqrt{x}\big)\bigg\},
\end{align}
the precise sense of which is quantified in the following proposition. 

\begin{proposition}\label{prop:high_prob_E}
There exists a large enough absolute constant $C>0$ in the definition of $E(x)$ such that for any $x\geq 1$, $\Prob(E(x))\geq 1- e^{-x}$.
\end{proposition}
\begin{proof}
Let $E_1(x),E_2(x),E_3(x)$ be the events in the definition of $E(x)$, namely
\begin{align}\label{ineq:high_prob_E_0}
E_1(x)& = \bigg\{ \sup_{v \in T}\pnorm{G v}{}\leq C\cdot r(T)\big(\sqrt{n}+\sqrt{ d(T) }+\sqrt{x} \bigg)\bigg\},\nonumber\\
E_2(x)& = \bigg\{ \bigabs{\pnorm{g}{}-\sqrt{n}}\leq C\sqrt{x} \bigg\},\nonumber\\
E_3(x)& = \bigg\{\sup_{v \in T}\abs{\iprod{h}{v}}\leq C\cdot r(T)\big(\sqrt{d(T)}+\sqrt{x}\big)\bigg\}.
\end{align}	
We shall prove that $\Prob(E_i(x))\geq 1-e^{-x}$ for $i=1,2,3$. 

\noindent (\textbf{$E_1(x)$ term}). We first derive an estimate for $\E \sup_{v \in T}\pnorm{G v}{}=\E \sup_{u \in B_n, v \in T}\iprod{u}{Gv}$ via the standard Slepian-Fernique Gaussian comparison inequality. Let 
\begin{align*}
X(u,v)\equiv \iprod{u}{Gv},\quad Y(u,v)\equiv \sqrt{2}\cdot \big(r(T)\iprod{g}{u}+\iprod{h}{v}\big).
\end{align*}
Then for $(u,v),(u',v') \in B_n\times T$,
\begin{align*}
&\E\big(X(u,v)-X(u',v')\big)^2 = \E\big(\iprod{u}{Gv}-\iprod{u'}{Gv'}\big)^2 \\
&= \E \bigg(\sum_{i \in [n], j \in [p]} G_{ij}(u_iv_j-u_i'v_j')\bigg)^2 = \sum_{i \in [n], j \in [p]} (u_iv_j-u_i'v_j')^2 \\
&\leq 2\cdot \big(r^2(T)\pnorm{u-u'}{}^2+\pnorm{v-v'}{}^2\big),
\end{align*}
and
\begin{align*}
\E\big(Y(u,v)-Y(u',v')\big)^2 &= 2 \E\big(r(T)\iprod{g}{u-u'}+\iprod{h}{v-v'} \big)^2\\
& =  2\cdot \big(r^2(T)\pnorm{u-u'}{}^2+\pnorm{v-v'}{}^2\big). 
\end{align*}
This means for $(u,v),(u',v') \in B_n\times T$,
\begin{align*}
\E\big(X(u,v)-X(u',v')\big)^2 = \E\big(\iprod{u}{Gv}-\iprod{u'}{Gv'}\big)^2\leq \E\big(Y(u,v)-Y(u',v')\big)^2,
\end{align*}
and so the classical Slepian-Fernique Gaussian comparison inequality (cf. \cite[Theorem 7.2.11]{vershynin2018high}) yields that
\begin{align}\label{ineq:high_prob_E_1}
\E \sup_{v \in T}\pnorm{G v}{} &= \E \sup_{u \in B_n, v \in T} X(u,v)\leq \E \sup_{u \in B_n, v \in T} Y(u,v)\nonumber\\
& =\sqrt{2}\Big(r(T) \E \pnorm{g}{}+ \E \sup_{v \in T} \iprod{h}{v}\Big)\leq C\cdot r(T)\big(\sqrt{n}+\sqrt{d(T)}\big).
\end{align}
Next by viewing $G\mapsto \sup_{v \in T}\pnorm{G v}{}$ as a $r(T)$-Lipschitz map on $\R^{n\times p}$, using Gaussian concentration inequality we have with probability $1-e^{-x}$,
\begin{align}\label{ineq:high_prob_E_2}
\sup_{v \in T}\pnorm{G v}{}\leq \E \sup_{v \in T}\pnorm{G v}{}+ C\cdot r(T)\sqrt{x}.
\end{align}
Combining (\ref{ineq:high_prob_E_1})-(\ref{ineq:high_prob_E_2}) shows that $\Prob(E_1(x))\geq 1-e^{-x}$.

\noindent (\textbf{$E_2(x)$ term}). As $g \mapsto \pnorm{g}{}$ is $1$-Lipschitz, we have by Gaussian concentration, with probability at least $1-e^{-x}$,
\begin{align}\label{ineq:high_prob_E_3}
\bigabs{\pnorm{g}{}-\E \pnorm{g}{}}\leq \sqrt{x}. 
\end{align}
On the other hand, by Gaussian-Poincar\'e inequality, $(\E \pnorm{g}{})^2\leq \E \pnorm{g}{}^2\leq (\E \pnorm{g}{})^2+1$, which implies
\begin{align}\label{ineq:high_prob_E_4}
\bigabs{\E\pnorm{g}{}-\sqrt{n}}\leq \sqrt{n}-\sqrt{n-1}\leq 1. 
\end{align}
The claim for $E_2(x)$ now follows by combining (\ref{ineq:high_prob_E_3})-(\ref{ineq:high_prob_E_4}), upon noting $x\geq 1$ and possibly adjusting constants.

\noindent (\textbf{$E_3(x)$ term}). As $h\mapsto \sup_{v \in T}\abs{\iprod{h}{v}}$ is $r(T)$-Lipschitz, by Gaussian concentration, with probability at least $1-e^{-x}$,
\begin{align*}
\sup_{v \in T}\abs{\iprod{h}{v}}\leq w(T)+ C\cdot r(T)\sqrt{x} \leq C\cdot r(T)\big(\sqrt{d(T)}+\sqrt{x}\big),
\end{align*}
as desired.
\end{proof}

Next we quantify the error of compactification for $\mathscr{E}_\pm$.

\begin{proposition}\label{prop:primal_cost_err}
Suppose $0 \in T$. For $L_w(x)$ defined in (\ref{def:L_Delta}), and $L_u>1$, let 
\begin{align}\label{def:rem}
\rem(x,L_u)\equiv 2\sqrt{n} L_w^3(x)/L_u+\big(\sqrt{n}L_w^2(x)/L_u\big)^2.
\end{align}
Then on the event $E(x)$,
\begin{align*}
&\bigabs{\mathscr{E}_\pm - \sup_{\substack{v \in T, w \in B_n(L_w(x))}}\inf_{u \in B_n(L_u)}\mathfrak{h}_\pm(u,v,w)}\leq  \rem(x,L_u).
\end{align*}
\end{proposition}
\begin{proof}
First note that
\begin{align}\label{ineq:gordon_cost_0}
\mathscr{E}_\pm& = \sup_{v \in T} \Big\{ \pm\pnorm{G_n v}{}^2 \mp \zeta\pnorm{v}{}^2 \Big\} = \sup_{v \in T,w \in \R^n}\inf_{u \in \R^n} \mathfrak{h}_\pm(u,v,w).
\end{align}
Any pair of maximizers $(\tilde{v}_1,\tilde{w}_1)$ of the above max-min problem must satisfy $\tilde{w}_1=G_n\tilde{v}_1$, so on $E(x)$, $
\pnorm{\tilde{w}_1}{}\leq \pnorm{G_n \tilde{v}_1}{}\leq L_w(x)$. This means on $E(x)$, for any $L_u >1$, 
\begin{align}\label{ineq:gordon_cost_4}
\mathscr{E}_\pm&= \sup_{v \in T,w \in B_n(L_w(x))}\inf_{u \in \R^n}\quad \mathfrak{h}_\pm(u,v,w)\nonumber\\
&\leq \sup_{v \in T,w \in B_n(L_w(x))}\inf_{u \in B_n(L_u)}\Big\{\pm\pnorm{w}{}^2 \mp \zeta\pnorm{v}{}^2 + \iprod{u}{G_n v-w}\Big\}.
\end{align}
Any pair of maximizers $(\tilde{v}_2,\tilde{w}_2)$ of the above max-min problem must satisfy
\begin{align*}
\pnorm{G_n \tilde{v}_2-\tilde{w}_2}{\infty}\leq L_u^{-1}\big(\pnorm{\tilde{w}_2}{}^2+\zeta\pnorm{\tilde{v}_2}{}^2\big)\leq  L_w^2(x)/L_u,
\end{align*}
 otherwise the cost optimum would be non-positive, contradicting to $\mathscr{E}_\pm\geq 0$ due to $0\in T$. This means that on $E(x)$,
\begin{align}\label{ineq:gordon_cost_1}
 &\hbox{RHS of (\ref{ineq:gordon_cost_4})} = \sup_{\substack{v \in T,w \in B_n(L_w(x)),\\ \pnorm{G_nv-w}{\infty}\leq L_w^2(x)/L_u }}\inf_{u \in B_n(L_u)}\Big\{\pm\pnorm{w}{}^2 \mp \zeta\pnorm{v}{}^2 + \iprod{u}{G_n v-w}\Big\}\nonumber\\
 & = \sup_{\substack{v \in T,w \in B_n(L_w(x)),\\ \pnorm{G_nv-w}{\infty}\leq L_w^2(x)/L_u }}\inf_{u \in B_n(L_u)}\Big\{ \pm \pnorm{G_n v}{}^2\mp \zeta\pnorm{v}{}^2 + \mathsf{L}_\pm(u,v,w)\Big\},
\end{align} 
where
\begin{align*}
\mathsf{L}_\pm(u,v,w)\equiv \pm 2\iprod{ G_n v}{w-G_n v} \pm \pnorm{w-G_n v}{}^2 + \iprod{u}{G_n v-w}.
\end{align*}
By choosing $u=0$ in the infimum, and using the bounds
\begin{align*}
\bigabs{\iprod{ G_n v}{w-G_nv}}&\leq L_w(x)\pnorm{w-G_nv}{}\leq \sqrt{n} L_w^3(x)/L_u,\\
\pnorm{w-G_nv}{}^2&\leq  \big(\sqrt{n}L_w^2(x)/L_u\big)^2
\end{align*}
under the prescribed $\ell_\infty$ constraint $\pnorm{G_n v-w}{\infty}\leq L_w^2(x)/L_u$, we conclude that on the event $E(x)$, 
\begin{align}\label{ineq:gordon_cost_3}
\hbox{RHS of (\ref{ineq:gordon_cost_1})}&\leq \sup_{v \in T} \Big\{ \pm\pnorm{G_n v}{}^2 \mp \zeta\pnorm{v}{}^2 \Big\}+ \rem(x,L_u) \nonumber\\
&= \mathscr{E}_\pm + \rem(x,L_u). 
\end{align}
Here the last identity uses (\ref{ineq:gordon_cost_0}). Combining (\ref{ineq:gordon_cost_4})-(\ref{ineq:gordon_cost_3}) we conclude the desired inequality. 
\end{proof}

The following proposition is a direct consequence of the Gaussian min-max Theorem \ref{thm:CGMT}.

\begin{proposition}\label{prop:compact_cgmt}
Let $L_w(x)$ be defined in (\ref{def:L_Delta}) and $L_u>1$. For any $z \in \R$, 
\begin{align*}
\Prob\bigg(\sup_{\substack{v \in T, w \in B_n(L_w(x))}}\inf_{u \in B_n(L_u)}\mathfrak{h}_\pm(u,v,w)\geq z\bigg)  \leq 2\Prob\bigg( \sup_{\substack{v \in T, w \in B_n(L_w(x))}}\inf_{u \in B_n(L_u)}\mathfrak{l}_\pm(u,v,w)\geq z\bigg).
\end{align*}
\end{proposition}

Now we will relate Gordon costs to the $\mathsf{F}_\pm$ functions defined in (\ref{def:F}). 

\begin{proposition}\label{prop:gordon_cost_F}
Suppose $0 \in T$. Let $L_w(x), \Delta(x)$ be defined as in (\ref{def:L_Delta}) with $K_w$ being a large enough absolute constant. Suppose
\begin{align*}
L_u > \max \big\{1,4L_w(x)\}.
\end{align*}
Then on the event $E(x)$,
\begin{align*}
\sup_{\substack{v \in T,w \in B_n(L_w(x))}}\inf_{u \in B_n(L_u)}\quad \mathfrak{l}_\pm(u,v,w)\leq \mathsf{F}_\pm(h)+ \Delta(x).
\end{align*}
\end{proposition}
\begin{proof}[Proof of Proposition \ref{prop:gordon_cost_F}: the case of $(\mathfrak{l}_+,\mathsf{F}_+)$]
	
The proof is divided into two steps.

\noindent (\textbf{Step 1}). In this step, we shall prove that for $K_w$ (in the definition of $L_w(x)$) chosen as a large enough absolute constant, on the event $E(x)$,
\begin{align}\label{ineq:gordon_cost_right_1}
&\sup_{v \in T,w \in B_n(L_w(x))}\inf_{u \in B_n(L_u)}\mathfrak{l}_+(u,v,w) \nonumber\\
&=  \sup_{\alpha \in \Lambda_T, w \in \R^n} \{\pnorm{w}{}^2-\zeta\alpha^2\},\quad \hbox{s.t. }\sup_{v \in T, \pnorm{v}{}=\alpha }\langle h, v\rangle \geq \pnorm{\alpha g-\sqrt{n} w}{}.
\end{align}
To prove (\ref{ineq:gordon_cost_right_1}), we start by noting that
\begin{align}\label{ineq:gordon_cost_left_l_reduction}
& \sup_{\substack{v \in T, w \in B_n(L_w(x))}}\inf_{u \in B_n(L_u)}\quad \mathfrak{l}_+(u,v,w)\nonumber\\
&=\sup_{\substack{v \in T, w \in B_n(L_w(x))}}\inf_{u \in B_n(L_u)}\bigg\{ n^{-1/2}\Big(\pnorm{v}{} \iprod{g}{u}+ \pnorm{u}{}\iprod{h}{v}\Big) + \pnorm{w}{}^2-\zeta\pnorm{v}{}^2-\iprod{u}{w} \bigg\}\nonumber\\
& = \sup_{\substack{v \in T, w \in B_n(L_w(x))}}\inf_{\beta \in [0,L_u]} \bigg\{\frac{\beta}{\sqrt{n}} \iprod{h}{v}-\beta \biggpnorm{ \pnorm{v}{} \frac{g}{\sqrt{n}}-w }{}+\pnorm{w}{}^2- \zeta\pnorm{v}{}^2   \bigg\}\nonumber\\
& = \sup_{\substack{\alpha \in \Lambda_T, v \in T, \pnorm{v}{}=\alpha,\\  w \in B_n(L_w(x))}}\inf_{\beta \in [0,L_u]} \bigg\{\frac{\beta}{\sqrt{n}}\Big(\langle h, v\rangle-\pnorm{ \alpha g-\sqrt{n}w }{}\Big)+ \pnorm{w}{}^2- \zeta \alpha^2  \bigg\}.
\end{align}
Let us now drop the upper bound constraint on $\beta$ for the right hand side of (\ref{ineq:gordon_cost_left_l_reduction}). Suppose $(\tilde{\alpha},\tilde{v},\tilde{w},\tilde{\beta}=L_u)$ is a saddle point for the max-min problem (\ref{ineq:gordon_cost_left_l_reduction}). Note that $\tilde{\beta}=L_u$ implies $\iprod{h}{\tilde{v}}\leq \pnorm{\tilde{\alpha}g-\sqrt{n}\tilde{w}}{}$. On the event $E(x)$, let
\begin{align*}
0\leq \tau\equiv n^{-1/2}\pnorm{\tilde{\alpha}g-\sqrt{n}\tilde{w}}{}-n^{-1/2}\iprod{h}{\tilde{v}}\leq 2 L_w(x).
\end{align*}
This means by writing $\tilde{w}=\tilde{\alpha}(g/\sqrt{n})+\big(\tau+n^{-1/2}\iprod{h}{\tilde{v}}\big)\cdot \tilde{r}$ for some $\pnorm{\tilde{r}}{}=1$, 
\begin{align}\label{ineq:gordon_cost_F_1}
\hbox{RHS of (\ref{ineq:gordon_cost_left_l_reduction})}&= \tau^2 - 2 \tau\cdot \bigg(\frac{L_u}{2}- \frac{\tilde{\alpha}}{\sqrt{n}}\iprod{g}{\tilde{r}}-\frac{1}{\sqrt{n}}\iprod{h}{\tilde{v}}\bigg) \nonumber\\
&+\frac{1}{n} \bigpnorm{ \tilde{\alpha} g + \iprod{h}{\tilde{v}}\tilde{r} }{}^2 - \zeta\alpha^2.
\end{align}
As 
\begin{align*}
\frac{L_u}{2}-\frac{\tilde{\alpha}}{\sqrt{n}}\iprod{g}{\tilde{r}}- \frac{1}{\sqrt{n}}\langle h, \tilde{v}\rangle> \frac{L_u}{2}- L_w(x) > L_w(x),
\end{align*}
the right hand side of (\ref{ineq:gordon_cost_F_1}) attains maximum for $\tau=0$ over the range $\tau \in [0, 2L_w(x)]$. This means that any $(\tilde{\alpha},\tilde{v},\tilde{w},\tilde{\beta}<L_u)$ is also a saddle point for the max-min problem (\ref{ineq:gordon_cost_right_1}) on $E(x)$. Consequently, on the event $E(x)$, the right hand side of (\ref{ineq:gordon_cost_right_1}) is the same with the infimum over $\beta \in [0,L_u]$ replaced by the infimum over $\beta \geq 0$. This means
\begin{align*}
\hbox{RHS of (\ref{ineq:gordon_cost_left_l_reduction})}&=\sup_{\alpha \in \Lambda_T, w \in B_n(L_w(x))} \{\pnorm{w}{}^2-\zeta \alpha^2\},\quad \hbox{s.t. }\sup_{v \in T, \pnorm{v}{}=\alpha }\langle h, v\rangle \geq \pnorm{\alpha g-\sqrt{n} w}{}.
\end{align*}
The constraint on $w$ can be removed for free on $E(x)$ because the constraint on the right hand side of the above display entails that $\pnorm{w}{}\leq \alpha(\pnorm{g}{}/\sqrt{n})+ \sup_{v \in T, \pnorm{v}{}=\alpha} \abs{\iprod{h}{v}}\leq L_w(x)$. This proves (\ref{ineq:gordon_cost_right_1}).

\noindent (\textbf{Step 2}).  Solving the optimization (\ref{ineq:gordon_cost_right_1}), on the event $E(x)$, 
\begin{align*}
&\sup_{ \substack{v \in T, w \in B_n(L_w(x))}}\inf_{u \in B_n(L_u)}\mathfrak{l}_+(u,v,w)\leq \sup_{\alpha \in \Lambda_T} \bigg\{\bigg(\frac{\pnorm{g}{}}{\sqrt{n}}\alpha+ \frac{1}{\sqrt{n}}\sup_{v \in T, \pnorm{v}{}=\alpha} \langle h, v\rangle \bigg)^2-\zeta\alpha^2\bigg\}.
\end{align*}
The claim follows by using that on $E(x)$,
\begin{align*}
&\biggabs{\bigg(\frac{\pnorm{g}{}}{\sqrt{n}}\alpha+ \frac{1}{\sqrt{n}}\sup_{v \in T, \pnorm{v}{}=\alpha} \langle h, v\rangle \bigg)^2-\bigg(\alpha+ \frac{1}{\sqrt{n}}\sup_{v \in T, \pnorm{v}{}=\alpha} \langle h, v\rangle \bigg)^2}\\
&\leq \alpha \biggabs{\frac{\pnorm{g}{}}{\sqrt{n}}-1}\cdot \bigg[\alpha\bigg(\frac{\pnorm{g}{}}{\sqrt{n}}+1\bigg)+\frac{2\abs{\sup_{v \in T}\iprod{h}{v}} }{\sqrt{n}}\bigg] \\
&\leq C \cdot r^2(T)\cdot\sqrt{\frac{x}{n}}\cdot \bigg(1+\sqrt{\frac{x}{n}}+\sqrt{\frac{d(T)}{n}} \bigg)= \Delta(x).
\end{align*}
The proof is complete. 
\end{proof}

\begin{proof}[Proof of Proposition \ref{prop:gordon_cost_F}: the case of $(\mathfrak{l}_-,\mathsf{F}_-)$]
Using the same calculation as in Step 1 for the case $(\mathfrak{l}_+,\mathsf{F}_+)$ , on the event $E(x)$, 
\begin{align*}
&\sup_{ \substack{v \in T, w \in B_n(L_w(x))}}\inf_{u \in B_n(L_u)}\mathfrak{l}_-(u,v,w)\leq \sup_{\alpha \in \Lambda_T} \bigg\{\zeta\alpha^2-\bigg(\frac{\pnorm{g}{}}{\sqrt{n}}\alpha- \frac{1}{\sqrt{n}}\sup_{v \in T, \pnorm{v}{}=\alpha} \langle h, v\rangle\bigg)_+^2\bigg\}.
\end{align*}
Finally proceeding as in the proof of Step 2 for the case $(\mathfrak{l}_+,\mathsf{F}_+)$, we may replace the term $\pnorm{g}{}/\sqrt{n}$ by $1$ on $E(x)$ at the price of a similar error term.
\end{proof}

We also need a variance bound for $\mathsf{F}_\pm$. Recall $\err_n(T)$ defined in (\ref{def:err}).
\begin{proposition}\label{prop:var_F}
There exists some absolute constant $C>0$ such that
\begin{align*}
 \var\big(\mathsf{F}_\pm(h)\big)\leq C\cdot \err_n^2(T),
\end{align*}
where  $h\sim \mathcal{N}(0,I_p)$. 
\end{proposition}
\begin{proof}
Clearly the functions $h\mapsto \mathsf{F}_\pm(h)$ are absolutely continuous. Note that for any $h_1,h_2 \in \R^p$,
\begin{align*}
&\abs{\mathsf{F}_+(h_1)-\mathsf{F}_+(h_2)}\leq \sup_{\alpha \in \Lambda_T} \frac{1}{\sqrt{n}} \bigg\{\, \biggabs{\sup_{v \in T, \pnorm{v}{}=\alpha} \iprod{h_1}{v}-\sup_{v \in T, \pnorm{v}{}=\alpha} \iprod{h_2}{v}}\\
&\qquad \times \biggabs{2\sqrt{\abs{\zeta}}\alpha+ \frac{1}{\sqrt{n}}\sup_{v \in T, \pnorm{v}{}=\alpha} \iprod{h_1}{v}+ \frac{1}{\sqrt{n}}\sup_{v \in T, \pnorm{v}{}=\alpha} \iprod{h_2}{v}} \bigg\}\\
&\leq \frac{1}{\sqrt{n}} \pnorm{h_1-h_2}{}\cdot r^2(T)\cdot\bigg(2\sqrt{\abs{\zeta}}+ \frac{1}{r(T)\sqrt{n}}\Big\{\bigabs{\sup_{v \in T}\iprod{h_1}{v}}+\bigabs{\sup_{v \in T}\iprod{h_2}{v}}\Big\}\bigg).
\end{align*}
A similar bound holds for $\abs{\mathsf{F}_-(h_1)-\mathsf{F}_-(h_2)}$. This means
\begin{align*}
\pnorm{\nabla \mathsf{F}_\pm(h)}{}\leq C\cdot\frac{ r^2(T)}{\sqrt{n}}\cdot\bigg(1+\sqrt{\abs{\zeta}}+\frac{  \abs{ \sup_{v \in T}\iprod{h}{v}} }{r(T)\sqrt{n}}\bigg).
\end{align*}
Now Gaussian-Poincar\'e inequality yields that
\begin{align*}
\var\big(\mathsf{F}_\pm(h)\big)\leq \E \pnorm{\nabla \mathsf{F}_\pm(h)}{}^2 \leq C\cdot \frac{r^4(T) }{n}\bigg(1\vee\abs{\zeta}\vee\frac{d(T)}{n}\bigg)=C\cdot \err_n^2(T),
\end{align*}
as desired. 
\end{proof}

\subsection{A rough comparison inequality}

With the above preparations we may now prove a rough comparison inequality. 

\begin{theorem}\label{thm:exp_rough_bound}
Suppose $0 \in T$. For any $z \in \R$ and $x\geq 1$, 
\begin{align*}
\Prob\big( \mathscr{E}_\pm \geq z \big)&\leq 2\cdot \bm{1}\big( \E\mathsf{F}_\pm(h) \geq z - C\cdot x\cdot  \err_n(T)\big)+ \frac{C}{x^2},
\end{align*}
where $h\sim \mathcal{N}(0,I_p)$. 
\end{theorem}

\begin{proof}
We will only prove the claim for $\mathscr{E}_+$, as the claim for $\mathscr{E}_-$ follows from completely similar arguments. Combining the proceeding Propositions \ref{prop:primal_cost_err}-\ref{prop:gordon_cost_F}, we have for any $z \in \R$,
\begin{align*}
\Prob\big( \mathscr{E}_+\geq z \big)&\leq \Prob \bigg( \sup_{v \in T,w \in B_n(L_w(x))}\inf_{u \in B_n(L_u)}\mathfrak{h}_+(u,v,w) \geq z- \rem(x,L_u)\bigg)+\Prob(E(x)^c)\\
&\leq 2\Prob\bigg( \sup_{v \in T,w \in B_n(L_w(x))}\inf_{u \in B_n(L_u)}\mathfrak{l}_+(u,v,w)\geq z-\rem(x,L_u)\bigg)+\Prob(E(x)^c)\\
&\leq 2\Prob\bigg(\mathsf{F}_+(h) \geq z-\rem(x,L_u) -\Delta(x)  \bigg)+2 \Prob(E(x)^c).
\end{align*}
As the left hand side and the far right hand side of the above inequality do not depend on $L_u$, we may let $L_u \to \infty$ to kill the term $\rem(x,L_u)$. Now using Proposition \ref{prop:var_F}, it is easy to see that
\begin{align*}
\Prob\big(\mathsf{F}_+(h) \geq z -\Delta(x)  \big)\leq \bm{1}\big(\E\mathsf{F}_+(h) \geq z - C\cdot x\cdot \err_n(T)-\Delta(x)  \big)+ \frac{C}{x^2}.
\end{align*}
Finally using that $\Prob(E(x)^c)\leq e^{-x}$ for $x\geq 1$ (by Proposition \ref{prop:high_prob_E}), and that $\Delta(x)$ can be assimilated into the term before it with a possibly larger constant $C>0$, to conclude the inequality.
\end{proof}

\subsection{A concentration inequality}

\begin{proposition}\label{prop:cov_op_conc}
	Suppose $0 \in T$. There exists some universal constant $C>0$ such that for any $x\geq 1$, with probability at least $1-e^{-x}$, 
	\begin{align*}
	\abs{\mathscr{E}_\pm- \E \mathscr{E}_\pm }\leq C r^2(T)\cdot \bigg(\sqrt{1\vee \abs{\zeta}\vee \frac{d(T)}{n}}\sqrt{\frac{x}{n}}+\frac{x}{n}\bigg).
	\end{align*}
\end{proposition}

\begin{proof}
	We focus on $\mathscr{E}_+$. The proof is inspired by that of \cite[Theorem 5]{koltchinskii2017concentration}. Take any $1$-Lipschitz function $\varphi:\R_{\geq 0}\to [0,1]$ such that $\varphi|_{[0,1]}=1$ and $\varphi|_{[2,\infty)}=0$, and define $f: \R^{p\times n}\to \R$ by
	\begin{align*}
	\mathsf{f}_\delta(Z)\equiv \sup_{t \in T}\big(n^{-1}t^\top ZZ^\top t- \zeta \pnorm{t}{}^2\big)\cdot  \varphi\bigg(\delta^{-1} \sup_{t \in T}\big(n^{-1}t^\top ZZ^\top t- \zeta \pnorm{t}{}^2\big)\bigg). 
	\end{align*}
	\noindent (\textbf{Step 1}). We claim that there exists some universal constant $C>0$ such that
	\begin{align}\label{ineq:cov_op_conc_1}
	\pnorm{\mathsf{f}_\delta}{\lip}\leq C\cdot \frac{\sqrt{\delta}+\sqrt{\abs{\zeta}} r(T)}{\sqrt{n}}\cdot r(T).
	\end{align}
	Fix any $Z_1,Z_2 \in \R^{p\times n}$. We shall give a bound for $\abs{\mathsf{f}_\delta(Z_1)-\mathsf{f}_\delta(Z_2)}$. Let 
	\begin{align*}
	W_i\equiv n^{-1} Z_iZ_i^\top - \zeta I_p,\quad i=1,2,
	\end{align*}
	and we write for notational convenience that 
	\begin{align*}
	\pnorm{W}{T}\equiv \sup_{t \in T} t^\top Wt,\quad W \in \R^{p\times p}.
	\end{align*}
	We only need to consider the case where one of $\pnorm{W_1}{T}, \pnorm{W_2}{T}$ is $\leq 2\delta$; otherwise $\abs{\mathsf{f}_\delta(Z_1)-\mathsf{f}_\delta(Z_2)}=0$. Assume for simplicity that $\pnorm{W_1}{T}\leq 2\delta$. Then using that $\varphi$ is $1$-Lipschitz and bounded by $1$,
	\begin{align}\label{ineq:cov_op_conc_2}
	\abs{\mathsf{f}_\delta(Z_1)-\mathsf{f}_\delta(Z_2)}&=\bigabs{\pnorm{W_1}{T} \varphi\big(\delta^{-1} \pnorm{W_1}{T}\big)- \pnorm{W_2}{T} \varphi\big(\delta^{-1} \pnorm{W_2}{T}\big)}\nonumber\\
	&\leq \bigabs{ \pnorm{W_1}{T}\big[ \varphi\big(\delta^{-1} \pnorm{W_1}{T}\big)-\varphi\big(\delta^{-1} \pnorm{W_2}{T}\big) \big] }\nonumber\\
	&\qquad +\bigabs{ \big(\pnorm{W_1}{T}-\pnorm{W_2}{T}\big) \varphi\big(\delta^{-1} \pnorm{W_2}{T}\big) }\nonumber\\
	&\leq \bigg(\frac{ \pnorm{W_1}{T} }{\delta}+1\bigg) \pnorm{W_1-W_2}{T}\leq 3 \pnorm{W_1-W_2}{T}.  
	\end{align}
	On the other hand,
	\begin{align*}
	&\pnorm{W_1-W_2}{T}\\
	&\leq n^{-1}\pnorm{Z_1(Z_1-Z_2)^\top }{T}+n^{-1}\pnorm{(Z_1-Z_2)Z_2^\top }{T}\\
	&\leq n^{-1/2}\cdot \big(\pnorm{n^{-1} Z_1Z_1^\top}{T}^{1/2}+\pnorm{ n^{-1}Z_2Z_2^\top}{T}^{1/2}\big)\cdot \pnorm{(Z_1-Z_2)(Z_1-Z_2)^\top}{T}^{1/2}\\
	& \lesssim n^{-1/2}\cdot \big(\sqrt{\abs{\zeta}} r(T)+\pnorm{W_1}{T}^{1/2}+\pnorm{W_2}{T}^{1/2}\big)\cdot \pnorm{(Z_1-Z_2)(Z_1-Z_2)^\top}{T}^{1/2}\\
	&\lesssim n^{-1/2}\cdot \big(\sqrt{\abs{\zeta}} r(T)+\sqrt{\delta}+\pnorm{W_1-W_2}{T}^{1/2}\big)\cdot \pnorm{(Z_1-Z_2)(Z_1-Z_2)^\top}{T}^{1/2}.
	\end{align*}
	Solving the above quadratic inequality yields
	\begin{align}\label{ineq:cov_op_conc_3}
	\pnorm{W_1-W_2}{\op}&\lesssim \frac{\sqrt{\delta}+\sqrt{\abs{\zeta}} r(T)}{\sqrt{n}} \pnorm{(Z_1-Z_2)(Z_1-Z_2)^\top}{T}^{1/2}\nonumber\\
	&\qquad + \frac{1}{n}\pnorm{(Z_1-Z_2)(Z_1-Z_2)^\top}{T}.
	\end{align}
	Combining (\ref{ineq:cov_op_conc_2}) and (\ref{ineq:cov_op_conc_3}), and using that $\abs{\mathsf{f}_\delta(Z_1)-\mathsf{f}_\delta(Z_2)}\leq 2\delta$ by definition, 
	\begin{align*}
	&\abs{\mathsf{f}_\delta(Z_1)-\mathsf{f}_\delta(Z_2)}\\
	&\lesssim \bigg(\frac{\sqrt{\delta}+\sqrt{\abs{\zeta}} r(T)}{\sqrt{n}} \pnorm{(Z_1-Z_2)(Z_1-Z_2)^\top}{T}^{1/2}+ \frac{1}{n}\pnorm{(Z_1-Z_2)(Z_1-Z_2)^\top}{T}\bigg)\wedge \delta\\
	&\lesssim \frac{\sqrt{\delta}+\sqrt{\abs{\zeta}}r(T)}{\sqrt{n}} \pnorm{(Z_1-Z_2)(Z_1-Z_2)^\top}{T}^{1/2}\\
	& \leq  \frac{\sqrt{\delta}+\sqrt{\abs{\zeta}}r(T)}{\sqrt{n}}\cdot r(T)\cdot  \pnorm{Z_1-Z_2}{F},
	\end{align*}
	proving the claim (\ref{ineq:cov_op_conc_1}). 
	
	\noindent (\textbf{Step 2}). Now by Gaussian concentration for Lipschitz functions (see, e.g., \cite[Eqn. (1.9)]{ledoux2013probability}, or \cite[Theorem 2.2.6]{gine2015mathematical}), we have for any $x\geq 1$, on an event $E_1(x)$ with probability at least $1-e^{-x}$,
	\begin{align}\label{ineq:cov_op_conc_4}
	\bigabs{\mathsf{f}_\delta(G)- \med \big(\mathsf{f}_\delta(G)\big)}\leq  C\cdot (\sqrt{\delta}+\sqrt{\abs{\zeta}}r(T))\cdot r(T)\cdot \sqrt{\frac{x}{n}}.
	\end{align}
	Here $\med \big(\mathsf{f}_\delta(G)\big)$ is the median of $\mathsf{f}_\delta(G)$. Below we shall prove that the above display holds where $\mathsf{f}_\delta(G)$ is replaced by $\mathscr{E}_+=\sup_{t \in T}\big(n^{-1}t^\top GG^\top t-\zeta\pnorm{t}{}^2\big)=\sup_{t \in T}\big(\pnorm{G_n t}{}^2-\zeta \pnorm{t}{}^2\big)$ for an appropriate choice of $\delta$.

	First consider the right tail. On the event $E_1(x)\cap E_2(\delta)$, where $E_1(x)$ is defined in (\ref{ineq:high_prob_E_0}) and $E_2(\delta)\equiv \{\sup_{t \in T}\big(\pnorm{G_n t}{}^2- \zeta\pnorm{t}{}^2\big)\leq \delta\}$, using $
	\med \big(\mathsf{f}_\delta(G)\big)\leq \med (\mathscr{E}_+)$,
	we have the right tail control
	\begin{align*}
	&\mathscr{E}_+= \mathsf{f}_\delta(G)\leq \med(\mathscr{E}_+)+ C\cdot (\sqrt{\delta}+\sqrt{\abs{\zeta}}r(T))\cdot r(T)\cdot \sqrt{\frac{x}{n}}.
	\end{align*}
	By Proposition \ref{prop:high_prob_E}, with the choice
	\begin{align}\label{ineq:cov_op_conc_5}
	\delta\equiv \delta(x)\equiv C_0\cdot r^2(T) \bigg(1+ \abs{\zeta}+ \frac{d(T)}{n} +\frac{x}{n}\bigg)
	\end{align}
	for a large enough absolute constant $C_0>0$, we have $\Prob(E_2(\delta(x)))\geq 1-e^{-x}$. Now combining the above two displays, on the event $E_1(x)\cap E_2(\delta(x))$ with probability at least $1-2e^{-x}$, we have
	\begin{align}\label{ineq:cov_op_conc_6}
	\mathscr{E}_+\leq \med(\mathscr{E}_+)+C\cdot r^2(T)\bigg(\sqrt{1\vee \abs{\zeta}\vee  \frac{d(T)}{n}}+\sqrt{\frac{x}{n}}\bigg)\sqrt{\frac{x}{n}}.
	\end{align}
	For the left tail, choosing the same $\delta=\delta(x)$ as in (\ref{ineq:cov_op_conc_5}), as $\mathsf{f}_\delta(G) = \mathscr{E}_+$ on $E_2(\delta(x))$ and $x\geq 1$,
	\begin{align*}
	\Prob\big(\mathsf{f}_\delta(G)>\med(\mathscr{E}_+)\big)&\geq \Prob\big(\mathscr{E}_+>\med(\mathscr{E}_+), \mathscr{E}_+\leq \delta(x) \big)\\
	&\geq \Prob\big( \mathscr{E}_+>\med(\mathscr{E}_+)\big)-\Prob\big(\mathscr{E}_+>\delta(x)\big)\\
	&\geq \frac{1}{2}-e^{-x} \geq 0.13.
	\end{align*}
	Now using \cite[Lemma 2]{koltchinskii2017concentration} and (\ref{ineq:cov_op_conc_1}), and the fact that $\mathscr{E}_+\geq \mathsf{f}_\delta(G)$, we obtain lower tail estimate: for any $x\geq 1$, with probability at least $1-e^{-x}$, 
	\begin{align}\label{ineq:cov_op_conc_7}
	\mathscr{E}_+\geq \med(\mathscr{E}_+)-C\cdot r^2(T)\bigg(\sqrt{1\vee \abs{\zeta} \vee \frac{d(T)}{n}}+\sqrt{\frac{x}{n}}\bigg)\sqrt{\frac{x}{n}}.
	\end{align}
	Combining (\ref{ineq:cov_op_conc_6})-(\ref{ineq:cov_op_conc_7}) and adjusting constants, we have established for any $x\geq 1$, with probability at least $1-e^{-x}$, 
	\begin{align}\label{ineq:cov_op_conc_8}
	\abs{\mathscr{E}_+- \med(\mathscr{E}_+)}\leq C\cdot r^2(T)\bigg(\sqrt{1\vee \abs{\zeta} \vee \frac{d(T)}{n}}+\sqrt{\frac{x}{n}}\bigg)\sqrt{\frac{x}{n}}.
	\end{align}
	Integrating the above display, one obtain
	\begin{align}\label{ineq:cov_op_conc_9}
	\bigabs{\E \mathscr{E}_+- \med(\mathscr{E}_+)}\leq C\cdot \frac{r^2(T)}{\sqrt{n}}\sqrt{1\vee \abs{\zeta} \vee \frac{d(T)}{n}}.
	\end{align}
	The claimed concentration of $\mathscr{E}_+$ around its mean now follows by combining the above two displays (\ref{ineq:cov_op_conc_8})-(\ref{ineq:cov_op_conc_9}).
\end{proof}

\subsection{Proof of Theorems \ref{thm:square_process} and \ref{thm:general_bound_sp_dT}}\label{section:proof_square_process}

We need two simple facts.

\begin{lemma}\label{lem:upper_bound_E_T}
	Suppose $0 \in T$. It holds that
	\begin{align*}
	E_\ast(T)=E_{+,1}(T) \leq  r^2(T)\cdot \bigg(2\sqrt{\frac{d(T)}{n}}+ \frac{d(T)}{n}\bigg).
	\end{align*}
\end{lemma}
\begin{proof}
	Using $
	\sup_{t \in T/r(T), \pnorm{t}{}=\alpha} \iprod{h}{t}\leq \sup_{t \in T/r(T)} \iprod{h}{t}$, we have
	\begin{align*}
	E_\ast(T)/r^2(T)&\leq \E\sup_{\alpha \in [0,1]} \bigg\{\bigg(\alpha+ \frac{ \sup_{t \in T/r(T)} \iprod{h}{t} }{\sqrt{n}}\bigg)^2-\alpha^2\bigg\} \\
	&= \frac{2 }{\sqrt{n}} \E  \sup_{t \in T/r(T)} \iprod{h}{t}+ \frac{1}{n}\E \Big(\sup_{t \in T/r(T)} \iprod{h}{t}\Big)^2.
	\end{align*}
	Now using the simple fact 
	\begin{align*}
	\E \sup_{t \in T/r(T)} \iprod{h}{t}\leq \E^{1/2} \Big(\sup_{t \in T/r(T)} \iprod{h}{t}\Big)^2 = \sqrt{d(T)},
	\end{align*}
	we obtain the desired claim.
\end{proof}

\begin{lemma}\label{lem:F_pm_comp}
Recall $\mathsf{F}_\pm$ in (\ref{def:F}). Suppose $0\in T$ and $\zeta=1$. Then $\mathsf{F}_+\geq \mathsf{F}_-$.
\end{lemma}
\begin{proof}
Let $G(\alpha)\equiv n^{-1/2}\sup_{v \in T, \pnorm{v}{}=\alpha} \iprod{h}{v}$. Then 
\begin{align*}
\mathsf{F}_+=\sup_{\alpha \in \Lambda_T} \big\{(\alpha+G(\alpha))^2-\alpha^2\big\},\quad \mathsf{F}_-= \sup_{\alpha \in \Lambda_T} \big\{\alpha^2- (\alpha-G(\alpha))_+^2\big\}.
\end{align*}
As $0 \in \Lambda_T$, we only need to consider $\alpha \in \Lambda_T$ for which $G(\alpha)\geq 0$. We consider two cases below:
\begin{enumerate}
	\item If $\alpha\geq G(\alpha)\geq 0$, then $\mathsf{F}_+= \sup_{\alpha \in \Lambda_T} \big\{2\alpha G(\alpha)+G^2(\alpha)\big\}$ and $\mathsf{F}_-= \sup_{\alpha \in \Lambda_T} \big\{2\alpha G(\alpha)-G^2(\alpha)\big\}$, so $\mathsf{F}_+\geq \mathsf{F}_-$.
	\item If $G(\alpha)>\alpha\geq 0$, then $\mathsf{F}_+=\sup_{\alpha\in \Lambda_T} \big\{2\alpha G(\alpha)+G^2(\alpha)\big\}$ and $\mathsf{F}_-=\sup_{\alpha\in \Lambda_T} \alpha^2$. As $2\alpha G(\alpha)+G^2(\alpha)\geq \alpha^2$, we have $\mathsf{F}_+\geq \mathsf{F}_-$.
\end{enumerate}
The proof is complete.
\end{proof}

We are now in a position to prove Theorems \ref{thm:square_process} and \ref{thm:general_bound_sp_dT}.

\begin{proof}[Proof of Theorems \ref{thm:square_process} and \ref{thm:general_bound_sp_dT}]
To conclude the control of the expectation from Theorem \ref{thm:exp_rough_bound}, we start by taking $
z \equiv \E \mathsf{F}_+(h)+ 2C\cdot x\cdot  \err_n(T)$
in the right tail inequality in Theorem \ref{thm:exp_rough_bound}, which shows that for $x\geq 1$,
\begin{align*}
\Prob\Big(\mathscr{E}_+\geq  \E \mathsf{F}_+(h)+ 2C\cdot x\cdot \err_n(T)\Big)\leq \frac{C}{x^2}. 
\end{align*}
Integrating the tail yields that
\begin{align*}
\E \mathscr{E}_+ \leq \E \mathsf{F}_+(h)+C\cdot   \err_n(T).
\end{align*}
Theorem \ref{thm:square_process} now follows from the concentration estimate in Proposition \ref{prop:cov_op_conc}.

Theorem \ref{thm:general_bound_sp_dT} follows by applying Theorem \ref{thm:square_process} with $T\cup \{0\}$, upon (i) using Lemma \ref{lem:F_pm_comp} with $\mathsf{F}_+\geq \mathsf{F}_-$ and (ii) using Lemma \ref{lem:upper_bound_E_T} with
\begin{align*}
{\frac{1}{\sqrt{n}}\sqrt{1\vee \frac{d(T)}{n}}}\bigg/ \bigg( \sqrt{\frac{d(T)}{n}}+ \frac{d(T)}{n}\bigg) \lesssim \frac{1}{\sqrt{d(T)}}.
\end{align*}
The proof is complete. 
\end{proof}

\section{Proofs of Theorems \ref{thm:sqrt_process} and \ref{thm:sqrt_process_zetaS}}\label{section:proof_sqrt_process}

\subsection{Some notation}

For two compact sets $T\subset \R^p$ and $S \subset \R^n$,
\begin{align}
\mathscr{E}_\pm^\circ \equiv \sup_{v \in T} \pm\bigg(\sup_{s \in S}\iprod{s}{G_n v}- \zeta\pnorm{v}{}\bigg).
\end{align}
Let the primal and Gordon cost functions associated with $\mathscr{E}_\pm^\circ$ be
\begin{align}\label{def:gordon_cost_h_l_sqrt}
\mathfrak{h}_\pm^\circ(s,u,v,w)&\equiv \pm \iprod{s}{w} \mp \zeta\pnorm{v}{} + \iprod{u}{G_n v-w},\nonumber\\
\mathfrak{l}_\pm^\circ(s,u,v,w)& \equiv n^{-1/2}\pnorm{v}{} \iprod{g}{u}+ n^{-1/2}\pnorm{u}{}\iprod{h}{v}\pm \iprod{s}{w} \mp \zeta\pnorm{v}{}-\iprod{u}{w}.
\end{align}
We define
\begin{align}\label{def:err_sqrt}
\err_n^\circ(T,S)\equiv \frac{r(T)r(S)}{\sqrt{n}},
\end{align}
and
\begin{align}\label{def:F_sqrt}
\mathsf{F}_{+}^\circ(g,h)&\equiv \sup_{\alpha \in \Lambda_T,\beta \in \Lambda_S}\bigg\{ \frac{\alpha}{\sqrt{n}}\sup_{s \in S, \pnorm{s}{}=\beta}\iprod{g}{s}+ \frac{\beta}{\sqrt{n}} \sup_{v\in T,\pnorm{t}{}=\alpha} \iprod{h}{t}-\zeta\alpha\bigg\},\nonumber\\
\mathsf{F}_{-}^\circ(g,h)&\equiv \sup_{\alpha \in \Lambda_T}\inf_{\beta \in \Lambda_S}\bigg\{\zeta\alpha- \frac{\alpha}{\sqrt{n}}\sup_{s \in S, \pnorm{s}{}=\beta}\iprod{g}{s}+ \frac{\beta}{\sqrt{n}} \sup_{v\in T,\pnorm{t}{}=\alpha} \iprod{h}{t}\bigg\}.
\end{align}
Throughout this section, let for $x\geq 1$
\begin{align}\label{def:L_Delta_sqrt}
L_w^\circ(x)&\equiv  K_w \cdot  r(T)\bigg(1+\sqrt{ \frac{d(T)+x}{n} }\bigg), 
\end{align}
where $K_w>0$ is a large enough absolute constant. We typically work on the event $E^\circ(x)=E(x)$ as defined in (\ref{def:event_E}).

\subsection{Reduction via Gaussian min-max theorem}

First we establish an analogue of Proposition \ref{prop:primal_cost_err} for $\mathscr{E}_\pm^\circ$  that quantifies the error in compactification. 
\begin{proposition}\label{prop:primal_cost_err_sqrt}
Suppose $0 \in T$. For $L_w^\circ(x)$ defined in (\ref{def:L_Delta_sqrt}) and $L_u^\circ>1$, let
\begin{align}\label{def:rem_sqrt}
\rem^\circ (x,L_u^\circ)\equiv \sqrt{n} (\abs{\zeta}+r(S))^2 L_w^\circ(x)/L_u^\circ.
\end{align}
Then on the event $E^\circ(x)$,
\begin{align*}
&\bigabs{\mathscr{E}_+^\circ- \sup_{v \in T, s \in S, w \in B_n(L_w^\circ(x))} \inf_{u \in B_n(L_u^\circ)}  \mathfrak{h}_+^\circ(s,u,v,w)}\\
& \quad \vee \bigabs{\mathscr{E}_-^\circ- \sup_{v \in T, w \in B_n(L_w^\circ(x))} \inf_{u \in B_n(L_u^\circ),s \in S}  \mathfrak{h}_-^\circ(s,u,v,w)} \leq \rem^\circ (x,L_u^\circ).
\end{align*}
\end{proposition}
\begin{proof}
First consider $\mathscr{E}_+^\circ$. Note that
\begin{align*}
\mathscr{E}_+^\circ\equiv \sup_{v \in T, s \in S, w \in \R^n} \inf_{u \in \R^n} \mathfrak{h}_+^\circ(s,u,v,w).
\end{align*}
Using the same argument below (\ref{ineq:gordon_cost_0}), on the event $E^\circ(x)$, we may restrict the supremum over $w \in \R^n$ to $w \in B_n(L_w^\circ(x))$. So on $E^\circ(x)$, we have the following analogue of (\ref{ineq:gordon_cost_4}): for any $L_u^\circ>1$,
\begin{align*}
\mathscr{E}_+^\circ\leq \sup_{v \in T, s \in S, w \in B_n(L_w^\circ(x))} \inf_{u \in B_n(L_u^\circ)} \Big\{ \iprod{s}{w} - \zeta\pnorm{v}{} + \iprod{u}{G_n v-w}\Big\}.
\end{align*}
Using the same argument below (\ref{ineq:gordon_cost_4}), we may restrict further the range of $v,w$ to $\pnorm{G_nv-w}{\infty}\leq (\abs{\zeta}+r(S)) L_w^\circ(x)/L_u^\circ$, as $0 \in T$. So on $E^\circ(x)$,
\begin{align*}
\mathscr{E}_+^\circ &\leq \sup_{v \in T, s \in S, w \in B_n(L_w^\circ(x))} \inf_{u \in B_n(L_u^\circ)}  \mathfrak{h}_+^\circ(s,u,v,w)\\
&\leq   \sup_{\substack{v \in T, s \in S, w \in B_n(L_w^\circ(x)),\\ \pnorm{G_nv-w}{\infty}\leq (\abs{\zeta}+r(S)) L_w^\circ(x)/L_u^\circ }} \inf_{u \in B_n(L_u^\circ)} \Big\{ \iprod{s}{w} - \zeta\pnorm{v}{} + \iprod{u}{G_n v-w}\Big\}\\
&\stackrel{(\ast)}{\leq} \sup_{\substack{v \in T, s \in S, w \in B_n(L_w^\circ(x)),\\ \pnorm{G_nv-w}{\infty}\leq (\abs{\zeta}+r(S)) L_w^\circ(x)/L_u^\circ }} \Big\{ \iprod{s}{G_n v}- \zeta\pnorm{v}{} +\iprod{s}{w-G_nv} \Big\}\\
&\stackrel{(\ast\ast)}{\leq} \mathscr{E}_+^\circ + \sqrt{n} (\abs{\zeta}+r(S))^2 L_w^\circ(x)/L_u^\circ.
\end{align*}
Here in $(\ast)$ we take $u=0$ in the infimum and write $\iprod{s}{w}=\iprod{s}{G_nv}+\iprod{s}{w-G_nv}$, and in $(\ast\ast)$ we use the constraint $\pnorm{G_n v-w}{\infty}\leq (\abs{\zeta}+r(S))L_w^\circ(x)/L_u^\circ$.

The case $\mathscr{E}_-^\circ$ follows similarly by formally replacing $\sup_{s \in S}$ with $\inf_{s \in S}$.
\end{proof}

The next proposition is analogous to Proposition \ref{prop:compact_cgmt}, but now connects the primal cost $\mathfrak{h}_\pm^\circ$ to the Gordon cost $\mathfrak{l}_\pm^\circ$.

\begin{proposition}\label{prop:compact_cgmt_sqrt}
	Let $L_w^\circ(x)$ be defined in (\ref{def:L_Delta_sqrt}), and $L_u^\circ>1$. For any $z \in \R$, 
	\begin{align*}
	&\Prob\bigg(\sup_{\substack{v \in T, s \in S,\\w \in B_n(L_w^\circ(x))}}\inf_{u \in B_n(L_u^\circ)}\mathfrak{h}_+^\circ(s,u,v,w)\geq z\bigg)  \leq 2\Prob\bigg( \sup_{\substack{v \in T, s \in S \\w \in B_n(L_w^\circ(x))}}\inf_{u \in B_n(L_u^\circ)}\mathfrak{l}_+^\circ(s,u,v,w)\geq z\bigg),\\
	&\Prob\bigg(\sup_{\substack{v \in T,\\w \in B_n(L_w^\circ(x))}}\inf_{ \substack{s\in S, \\u \in B_n(L_u^\circ) }}\mathfrak{h}_-^\circ(s,u,v,w)\geq z\bigg)  \leq 2\Prob\bigg( \sup_{\substack{v \in T, \\w \in B_n(L_w^\circ(x))}}\inf_{ \substack{s\in S, \\u \in B_n(L_u^\circ) }}\mathfrak{l}_-^\circ(s,u,v,w)\geq z\bigg).
	\end{align*}
\end{proposition}

Then we relate $\mathfrak{l}_\pm^\circ$ to $\mathsf{F}^\circ_\pm$ in analogy to Proposition \ref{prop:gordon_cost_F}.
\begin{proposition}\label{prop:gordon_cost_F_sqrt}
	Suppose $0 \in T$. Let $L_w^\circ(x), \Delta^\circ(x)$ be defined as in (\ref{def:L_Delta_sqrt}) with $K_w$ being a large enough absolute constant. Suppose
	\begin{align*}
	L_u^\circ > \max \big\{1,r(S)\}.
	\end{align*}
	Then on the event $E^\circ(x)$,
	\begin{align*}
	 &\sup_{v \in T, s \in S, w \in B_n(L_w^\circ(x))} \inf_{u \in B_n(L_u^\circ)}  \mathfrak{l}_+^\circ(s,u,v,w)\leq \mathsf{F}^\circ_+(g,h),\\
	 &\sup_{v \in T, w \in B_n(L_w^\circ(x))} \inf_{u \in B_n(L_u^\circ), s \in S}  \mathfrak{l}_-^\circ(s,u,v,w)\leq \mathsf{F}^\circ_-(g,h).
	\end{align*}
\end{proposition}

\begin{proof}
First consider $\mathfrak{l}^\circ_+$. Using calculations similar to (\ref{ineq:gordon_cost_left_l_reduction}),
\begin{align}\label{ineq:gordon_cost_left_l_sqrt_1}
&\sup_{v \in T, s \in S, w \in B_n(L_w^\circ(x))} \inf_{u \in B_n(L_u^\circ)}  \mathfrak{l}_+^\circ(s,u,v,w)\nonumber\\
& = \sup_{\substack{s\in S, \alpha \in \Lambda_T, v \in T, \pnorm{v}{}=\alpha,\\  w \in B_n(L_w^\circ(x))}}\inf_{\beta \in [0,L_u^\circ]} \bigg\{\frac{\beta}{\sqrt{n}}\Big(\langle h, v\rangle-\pnorm{ \alpha g-\sqrt{n}w }{}\Big)+ \iprod{s}{w}-\zeta \alpha  \bigg\}.
\end{align}
First we drop the constraint on $\beta$. Suppose $(\tilde{s},\tilde{\alpha},\tilde{v},\tilde{w},\tilde{\beta}=L_u^\circ)$ is a saddle point for the above max-min problem. Then we must have $\iprod{h}{\tilde{v}}\leq \pnorm{\tilde{\alpha}g-\sqrt{n}\tilde{w}}{}$. Now on the event $E^\circ(x)$, let
\begin{align*}
0\leq \tilde{\tau}\equiv n^{-1/2}\pnorm{\tilde{\alpha}g-\sqrt{n}\tilde{w}}{}-n^{-1/2}\iprod{h}{\tilde{v}}\leq  2L_w^\circ(x).
\end{align*} 
This means by writing $\tilde{w}=\tilde{\alpha}(g/\sqrt{n})+\big(\tilde{\tau}+n^{-1/2}\iprod{h}{\tilde{v}}\big)\cdot \tilde{r}$ for some $\pnorm{\tilde{r}}{}=1$, 
\begin{align*}
\hbox{RHS of (\ref{ineq:gordon_cost_left_l_sqrt_1})}= \big(-L_u^\circ + \iprod{\tilde{s}}{\tilde{r}}\big)\tilde{\tau}+ \Big(n^{-1/2}\tilde{\alpha}\iprod{\tilde{s}}{g}+n^{-1/2}\iprod{h}{\tilde{v}}\iprod{\tilde{s}}{\tilde{r}}-\zeta \tilde{\alpha}\Big).
\end{align*}
Clearly for $L_u^\circ>r(S)$, the right hand side of the above display is maximized for $\tilde{\tau}=0$. This means that on $E^\circ(x)$, any $\beta \in (0,L_u^\circ)$ is also a saddle point for $(\tilde{s},\tilde{\alpha},\tilde{v},\tilde{w})$. Consequently, on $E^\circ(x)$, the constraint $\beta \in [0,L_u^\circ]$ can be replaced by $\beta \in [0,\infty)$. Moreover, the suprema over $w \in B_n(L_w^\circ (x))$ can be replaced by $w \in \R^n$ on $E^\circ(x)$.  Therefore, on $E^\circ(x)$,
\begin{align*}
&\sup_{v \in T, s \in S, w \in B_n(L_w^\circ(x))} \inf_{u \in B_n(L_u^\circ)}  \mathfrak{l}_+^\circ(s,u,v,w)\nonumber\\
& = \sup_{\substack{s\in S, \alpha \in \Lambda_T, w \in \R^n }} \iprod{s}{w}-\zeta \alpha, \quad \hbox{s.t. } \sup_{v\in T,\pnorm{v}{}=\alpha} \iprod{h}{v}\geq \pnorm{\alpha g-\sqrt{n}w}{}.
\end{align*} 
Now by writing $w = n^{-1/2}\alpha g+ r$, where $r \in   B_n\big(r_T(\alpha)\big)$ with $r_T(\alpha)\equiv n^{-1/2}\sup_{v\in T,\pnorm{v}{}=\alpha} \iprod{h}{v}$, the right hand side of the above display is equal to
\begin{align*}
\sup_{s \in S, \alpha \in \Lambda_T, r_T(\alpha)\geq 0} \frac{\alpha}{\sqrt{n}}\iprod{g}{s}+ \frac{\pnorm{s}{}}{\sqrt{n}} \sup_{v\in T,\pnorm{v}{}=\alpha} \iprod{h}{v}-\zeta\alpha \leq \mathsf{F}_+^\circ(g,h),
\end{align*}
proving the case for $\mathfrak{l}_+^\circ$.

For $\mathfrak{l}_-^\circ$, similar arguments as above show that on $E^\circ(x)$, 
\begin{align*}
&\sup_{v \in T, w \in B_n(L_w^\circ(x))} \inf_{u \in B_n(L_u^\circ), s \in S}  \mathfrak{l}_-^\circ(s,u,v,w)\\
& =\sup_{\substack{\alpha \in \Lambda_T, w \in \R^n }} \inf_{ s\in S } \big\{\zeta \alpha-\iprod{s}{w}\big\}, \quad \hbox{s.t. } \sup_{v\in T,\pnorm{v}{}=\alpha} \iprod{h}{v}\geq \pnorm{\alpha g-\sqrt{n}w}{}\\
& = \sup_{\alpha \in \Lambda_T,  r_T(\alpha)\geq 0}\sup_{r \in B_n(r_T(\alpha))} \inf_{s \in S} \big\{\zeta \alpha-n^{-1/2}\alpha \iprod{g}{s}-\iprod{s}{r}\big\}\\
&\stackrel{(\ast)}{\leq} \sup_{\alpha \in \Lambda_T} \inf_{s \in S} \sup_{ r \in B_n(r_T(\alpha)) }\big\{\zeta \alpha-n^{-1/2}\alpha \iprod{g}{s}-\iprod{s}{r}\big\}\\
& = \sup_{\alpha \in \Lambda_T}\inf_{s \in S} \big\{\zeta \alpha-n^{-1/2}\alpha \iprod{g}{s}+\pnorm{s}{} r_T(\alpha)\big\}=\mathsf{F}_-^\circ(g,h),
\end{align*}
as desired.
In ($\ast$) above, we drop the constraint $r_T(\alpha)\geq 0$ and use the trivial inequality $\sup_r \inf_s \leq \inf_s \sup_r$.
\end{proof}

Finally we provide a  concentration inequality for $\mathsf{F}_{\pm}^\circ$.

\begin{proposition}\label{prop:var_F_sqrt}
There exists some absolute constant $C>0$ such that for $x\geq 1$,
\begin{align*}
\Prob\big(\abs{(\mathrm{id}-\E)\mathsf{F}_{\pm}^\circ(g,h)}>C\cdot \err_n^\circ(T,S)\cdot \sqrt{x} \big)\leq e^{-x}.
\end{align*}
\end{proposition}
\begin{proof}
Note that for any $g_1,g_2\in \R^n,h_1,h_2\in \R^p$,
\begin{align*}
\abs{\mathsf{F}_{\pm}^\circ(g_1,h_1)-\mathsf{F}_{\pm}^\circ(g_2,h_2)}\leq n^{-1/2} r(T)r(S)\cdot\big(\pnorm{g_1-g_2}{}+\pnorm{h_1-h_2}{}\big). 
\end{align*}
The claim now follows from Gaussian concentration.
\end{proof}

\subsection{Proof of Theorem \ref{thm:sqrt_process}}

\begin{proof}[Proof of Theorem \ref{thm:sqrt_process}]
First consider $\mathscr{E}_+^\circ$. Combining the proceeding Propositions \ref{prop:primal_cost_err_sqrt}-\ref{prop:gordon_cost_F_sqrt}, we have for any $z \in \R$,
\begin{align*}
\Prob\big( \mathscr{E}_+^\circ\geq z \big)&\leq \Prob \bigg( \sup_{v \in T,s\in S, w \in B_n(L_w^\circ(x))}\inf_{u \in B_n(L_u^\circ)}\mathfrak{h}_+^\circ(s,u,v,w) \geq z- \rem^\circ(x,L_u^\circ)\bigg)+\Prob(E^\circ(x)^c)\\
&\leq 2\Prob\bigg( \sup_{v \in T, s\in S, w \in B_n(L_w^\circ(x))}\inf_{u \in B_n(L_u^\circ)}\mathfrak{l}_+^\circ(u,v,w)\geq z-\rem^\circ(x,L_u^\circ)\bigg)+\Prob(E^\circ(x)^c)\\
&\leq 2\Prob\big(\mathsf{F}_+^\circ(g,h) \geq z-\rem^\circ(x,L_u^\circ)  \big)+2 \Prob(E^\circ(x)^c).
\end{align*}
By letting $L_u \to \infty$ we may kill the term $\rem^\circ(x,L_u^\circ)$. Now using Proposition \ref{prop:var_F_sqrt},
\begin{align*}
\Prob\big(\mathsf{F}_+^\circ(g,h) \geq z  \big)\leq \bm{1}\big(\E\mathsf{F}_+^\circ(g,h) \geq z - C\cdot \err_n^\circ(T,S)\cdot \sqrt{x} \big)+e^{-x}.
\end{align*}
Finally using  $\Prob(E^\circ(x)^c)\leq e^{-x}$ for $x\geq 1$ and choosing $z\equiv \E\mathsf{F}_+^\circ(g,h)+ 2C\cdot \err_n(T,S)\cdot \sqrt{x} $ to conclude the right tail control of $\mathscr{E}_+^\circ$. The claim for the right tail control of $\mathscr{E}_-^\circ$ follows from completely similar arguments. 
\end{proof}

\subsection{Proof of Theorem \ref{thm:sqrt_process_zetaS}}

\begin{lemma}\label{lem:upper_bound_E_zeta_sqrt}
	Recall $\zeta_S= n^{-1/2}w_1(S)$. Suppose $0 \in T$. Then
	\begin{align*}
	\max_{*=\pm}E_{\ast,\zeta_S}^\circ(T,S)\leq \frac{w(T) r(S) }{\sqrt{n}}\cdot \bigg(1+\frac{1}{ \sqrt{d(T)} }\bigg).
	\end{align*}
\end{lemma}
\begin{proof}
	For $E_{+,\zeta_S}^\circ(T,S)$, by expanding $\sup_{s \in S, \pnorm{s}{}=\beta}$ to $\sup_{s \in S}$, we have
	\begin{align*}
	&E_{+,\zeta_S}^\circ(T,S)\leq \E \sup_{\alpha \in \Lambda_T,\beta \in \Lambda_S}\bigg\{ \frac{\alpha}{\sqrt{n}}\Big(\sup_{s \in S}\iprod{g}{s}-w_1(S) \Big)+ \frac{\beta}{\sqrt{n}} \sup_{v\in T,\pnorm{t}{}=\alpha} \iprod{h}{t}\bigg\}\\
	&\leq \frac{r(T)}{\sqrt{n}}\E \bigabs{\sup_{s \in S}\iprod{g}{s}-w_1(S)}+ \frac{r(S) w(T)}{\sqrt{n}}\leq \frac{r(S) w(T)}{\sqrt{n}}\cdot \bigg(1+\frac{1}{ \sqrt{d(T)} }\bigg). 
	\end{align*}
	Here the last inequality follows by using Gaussian-Poincar\'e inequality upon noting that the map $g\mapsto \sup_{s \in S}\iprod{g}{s}$ is $r(S)$-Lipschitz. 
	
	On the other hand, for $E_{-,\zeta_S}^\circ(T,S)$, 
	\begin{align*}
	&E_{-,\zeta_S}^\circ(T,S)\leq \E \sup_{\alpha \in \Lambda_T}\inf_{\beta \in \Lambda_S}\bigg\{\zeta_S\alpha- \frac{\alpha}{\sqrt{n}}\sup_{s \in S, \pnorm{s}{}=\beta}\iprod{g}{s}+ \frac{r(S)}{\sqrt{n}} \bigabs{\sup_{v\in T,\pnorm{t}{}=\alpha} \iprod{h}{t}}\bigg\}\\
	&\leq \E \sup_{\alpha \in \Lambda_T} \bigg\{\alpha\Big(\zeta_S - \frac{1}{\sqrt{n}}\sup_{s \in S }\iprod{g}{s}\Big)+ \frac{r(S)}{\sqrt{n}} \bigabs{\sup_{v\in T,\pnorm{t}{}=\alpha} \iprod{h}{t}}\bigg\}\\
	&\leq \frac{r(T)}{\sqrt{n}}\E \bigabs{\sup_{s \in S}\iprod{g}{s}-w_1(S)}+ \frac{r(S) w(T)}{\sqrt{n}}\leq \frac{r(S) w(T)}{\sqrt{n}}\cdot \bigg(1+\frac{1}{ \sqrt{d(T)} }\bigg),
	\end{align*}
	as desired. 
\end{proof}

\begin{proof}[Proof of Theorem \ref{thm:sqrt_process_zetaS}]
	Using Theorem \ref{thm:sqrt_process} and Lemma \ref{lem:upper_bound_E_zeta_sqrt}, with the desired probability,
	\begin{align*}
	\sup_{v \in T} \Big|\sup_{s \in S}\iprod{s}{G_n v}- \zeta_S\pnorm{v}{}\Big|\leq \frac{r(S)w(T)}{\sqrt{n}}\cdot \bigg(1+\frac{1}{ \sqrt{d(T)} }\bigg)+ C\cdot r(T)r(S) \sqrt{\frac{x}{n}}. 
	\end{align*}
	The claim follows.
\end{proof}

\section{Remaining proofs for Section \ref{section:application}}\label{section:proof_application}

\subsection{Proof of Theorem \ref{thm:DM}}\label{subsection:proof_DM}

\begin{proof}[Proof of Theorem \ref{thm:DM}]
	By \cite[Exercise 11.3.2-(b)]{vershynin2018high}, $B_n(r_-)\subset \mathrm{conv}(G_n T)\subset B_n(r_+)$ if and only if
	\begin{align}\label{ineq:DM_variational}
	r_- \pnorm{v}{}\leq \sup_{s \in T}\iprod{G_n s}{v}=\sup_{s \in T}\iprod{s}{G_n^\top v}\leq r_+ \pnorm{v}{},\quad \forall v \in B_n.
	\end{align}
	As $\mathrm{conv}(G_n T)=G_n(\mathrm{conv}(T))$, we may assume without loss of generality that $T$ is convex. We will use Theorem \ref{thm:sqrt_process} with the role of $n,p$ flipped (but the scaling $n^{-1/2}$ remains the same). To this end,  we first compute, with $\zeta_T = w_1(T)/\sqrt{n}$,
	\begin{align*}
	E_{+,\zeta_T}^\circ(B_n,T)&= \E \sup_{\beta \in [0,1]}\beta \sup_{\alpha \in \Lambda_T}\bigg\{ \frac{1}{\sqrt{n}}\sup_{t \in T, \pnorm{t}{}=\alpha}\iprod{h}{t}+ \frac{\alpha  \pnorm{g}{}}{\sqrt{n}} -\zeta_T\bigg\}\\
	& = \E \bigg\{ \sup_{\alpha \in \Lambda_T}\bigg( \frac{1}{\sqrt{n}}\sup_{t \in T, \pnorm{t}{}=\alpha}\iprod{h}{t}+\alpha\bigg)-\zeta_T  \bigg\}_+ + \bigo\bigg(\frac{r(T)}{\sqrt{n}}\bigg).
	\end{align*}
	As the map $h\mapsto \sup_{\alpha \in \Lambda_T}\big( n^{-1/2}\sup_{t \in T, \pnorm{t}{}=\alpha}\iprod{h}{t}+\alpha\big)$ is $n^{-1/2} r(T)$-Lipschitz, by Gaussian concentration, 
	\begin{align*}
	E_{+,\zeta_T}^\circ(B_n,T)  = \bigg\{\E \sup_{\alpha \in \Lambda_T}\bigg( \frac{1}{\sqrt{n}}\sup_{t \in T, \pnorm{t}{}=\alpha}\iprod{h}{t}+\alpha\bigg)-\frac{w_1(T)}{\sqrt{n}}  \bigg\}_+ + \bigo\bigg(\frac{r(T)}{\sqrt{n}}\bigg).
	\end{align*}
	Similarly we have
	\begin{align*}
	E_{-,\zeta_T}^\circ(B_n,T) = \bigg\{ \frac{w_1(T)}{\sqrt{n}} - \E\sup_{\alpha \in \Lambda_T}\bigg( \frac{1}{\sqrt{n}}\sup_{t \in T, \pnorm{t}{}=\alpha}\iprod{h}{t}-\alpha\bigg) \bigg\}_+ + \bigo\bigg(\frac{r(T)}{\sqrt{n}}\bigg).
	\end{align*}
	For notational simplicity, let us write $E_{\pm,\zeta_T}^\circ(B_n,T)  \equiv \mathcal{E}_{\pm}^\circ(T)+\bigo(r(T)/\sqrt{n})$. Now using Theorem \ref{thm:sqrt_process},  with probability at least $1-e^{-x}$,
	\begin{align*}
	& \sup_{v \in  B_n}\pm\bigg( \sup_{s \in T}\iprod{s}{G_n^\top v}- \frac{w_1(T)}{\sqrt{n} } \pnorm{v}{} \bigg)\leq  \mathcal{E}_{\pm}^\circ(T) + C\cdot r(T) \sqrt{ \frac{x}{n} }.
	\end{align*}
	The claim follows using the homogeneity in the suprema over $v \in B_n$. 
\end{proof}

\subsection{Proof of Proposition \ref{prop:spiked_cov}}\label{subsection:proof_spiked_all}

We shall first compute the Gaussian width of spherical sections of the ellipsoid associated with the spiked covariance model. 

\begin{lemma}\label{lem:gw_ellipsoid}
	Fix $\lambda\geq 0$ and an orthonormal system $\{\mathsf{v}_j: j \in [p]\}\subset \R^p$. Consider the spiked covariance $\Sigma=I_p+\lambda \sum_{j=1}^r \mathsf{v}_j\mathsf{v}_j^\top$. For $ h \in \R^p$, let $Z\in \R^p$ be defined by $Z_j\equiv \langle h,\mathsf{v}_j\rangle$ for $j \in [p]$, and let  $\delta_n\equiv {  \pnorm{Z_{(r:p]}}{}^2 }/{n}$, $\Xi_n\equiv n^{-1/2}\sqrt{1+\lambda}\cdot \pnorm{Z_{[1:r]}}{}$. Then
	\begin{align*}
	\sup_{\alpha \in [0,\sqrt{1+\lambda}] }\biggabs{\frac{1}{\sqrt{n}}\sup_{t \in \Sigma^{1/2}(B_p), \pnorm{t}{}=\alpha}\iprod{h}{t}-\sqrt{\delta_n}\cdot \Gamma_\lambda(\alpha)}= \bigo(\Xi_n),
	\end{align*}
	where  
	\begin{align}\label{def:Gamma_lambda}
	\Gamma_\lambda(\alpha)\equiv \sqrt{ (1\wedge \alpha^2) \Big(1+\frac{1}{\lambda}\Big) - \frac{\alpha^2}{\lambda}  },\quad \alpha \in [0,\sqrt{1+\lambda}].
	\end{align}
\end{lemma}
\begin{proof}
	Under the assumed spiked covariance model, 
	\begin{align}\label{ineq:spike_cov_01}
	\pnorm{\Sigma}{\op}=1+\lambda,\quad \Sigma^{1/2}=I_p+u_{\lambda}\sum\nolimits_{j \in [r]} \mathsf{v}_j\mathsf{v}_j^\top,
	\end{align}
	where  $u_{\lambda}\geq 0$ is the unique non-negative solution of $u_{\lambda}^2+2u_{\lambda}=\lambda$. For $\alpha \in [0,\sqrt{1+\lambda}]$, we start computing the Gaussian width
	\begin{align}\label{ineq:spike_cov_1}
	\sup_{v \in B_p, \pnorm{\Sigma^{1/2}v}{}=\alpha}\iprod{h}{\Sigma^{1/2}v} = \sup_{ \substack{v \in B_p,\\ \pnorm{v}{}^2+\lambda\sum_{j=1}^r \iprod{v}{\mathsf{v}_j}^2=\alpha^2} }\bigg(\iprod{h}{v}+u_{\lambda} \sum_{j=1}^r  \iprod{v}{\mathsf{v}_j}\iprod{h}{\mathsf{v}_j}\bigg). 
	\end{align}
	Using $Z$ instead of $h$, we have 
	\begin{align}\label{ineq:spike_cov_2}
	&\biggabs{ \sup_{ \substack{v \in B_p,\\ \pnorm{v}{}^2+\lambda \sum_{j=1}^r \iprod{v}{\mathsf{v}_j}^2=\alpha^2} }\bigg(\iprod{h}{v}+u_{\lambda} \sum_{j=1}^r  \iprod{v}{\mathsf{v}_j}\iprod{h}{\mathsf{v}_j}\bigg)-  \sup_{ \substack{v \in B_p,\\ \pnorm{v}{}^2+ \lambda \sum_{j=1}^r \iprod{v}{\mathsf{v}_j}^2=\alpha^2} }\iprod{h}{v} }\nonumber\\
	&\qquad \leq  \sup_{ {v \in B_p, \pnorm{v}{}^2+ \lambda \sum_{j=1}^r \iprod{v}{\mathsf{v}_j}^2=\alpha^2} }  u_{\lambda}\biggabs{\sum_{j=1}^r  \iprod{v}{\mathsf{v}_j}Z_j} \nonumber\\
	&\qquad =  \sup_{ \substack{b \in B_p, \pnorm{b}{}^2+ \lambda\sum_{j=1}^r  b_j^2=\alpha^2} } u_{\lambda} \biggabs{\sum_{j=1}^r  b_j Z_j} \nonumber \\
	&\qquad \leq \sup_{b_0 \in [\alpha/\sqrt{1+\lambda},1\wedge \alpha]}\sup_{ \substack{\pnorm{b}{}=b_0,\\ u_{\lambda}^2\sum_{j=1}^r  b_j^2\leq \alpha^2-b_0^2}}  u_{\lambda}\biggabs{\sum_{j=1}^r  b_j Z_j}\nonumber\\
	&\qquad \leq  \sup_{b_0 \in [\alpha/\sqrt{1+\lambda},1\wedge \alpha]} \sqrt{\alpha^2-b_0^2}\cdot  \pnorm{Z_{[1:r]}}{} \leq  \sqrt{\lambda}\cdot \pnorm{Z_{[1:r]}}{}.
	\end{align}
	On the other hand, we may compute 
	\begin{align}\label{ineq:spike_cov_3}
	&\sup_{ \substack{v \in B_p,\\ \pnorm{v}{}^2+\lambda \sum_{j=1}^r \iprod{v}{\mathsf{v}_j}^2=\alpha^2} }\iprod{h}{v}
	=\sup_{\gamma \in [\alpha/\sqrt{1+\lambda },1\wedge \alpha]} \sup_{\substack{\pnorm{v}{}=\gamma,\\  \sum_{j=1}^r  v_j^2=(\alpha^2-\gamma^2)/\lambda }}\iprod{Z}{v} \nonumber\\
	& = \sup_{\gamma \in [\alpha/\sqrt{1+\lambda},1\wedge \alpha]}\bigg\{\sqrt{\frac{\alpha^2-\gamma^2}{\lambda}}\cdot \pnorm{Z_{[1:r]}}{}+\sqrt{\gamma^2-\frac{\alpha^2-\gamma^2}{\lambda}}\cdot \pnorm{Z_{(r:p]}}{}\bigg\}\nonumber\\
	& = \sup_{\gamma \in [\alpha/\sqrt{1+\lambda},1\wedge \alpha]} \sqrt{\gamma^2-\frac{\alpha^2-\gamma^2}{\lambda}}\cdot \pnorm{Z_{(r:p]}}{}+ \bigo\big(\pnorm{Z_{[1:r]}}{} \big),
	\end{align}
	where the last inequality follows as $
	\sup_{\gamma \in [\alpha/\sqrt{1+\lambda},1\wedge \alpha]}\sqrt{(\alpha^2-\gamma^2)/{\lambda}} = \sqrt{{\alpha^2}/(1+\lambda)}\leq 1$. 
	The claim now follows by combining (\ref{ineq:spike_cov_1})-(\ref{ineq:spike_cov_3}).
\end{proof}

The next two analytic lemmas will be crucial in our calculations in the proof of Proposition \ref{subsection:proof_spiked_all}.
\begin{lemma}\label{lem:sup_H_bar}
Let
\begin{align*}
\overline{\mathsf{H}}_\delta(\eta)\equiv \Big(\sqrt{1+\lambda \eta}+\sqrt{\delta(1-\eta)}\Big)^2,\quad \eta \in [0,1].
\end{align*}
Recall $\overline{\Psi}_\delta$ defined in (\ref{def:psi_spiked_cov}). Then $\overline{\Psi}_\delta$ is non-decreasing on $[\sqrt{\delta},\infty)$ and
\begin{align*}
\max_{\eta \in [0,1]}\overline{\mathsf{H}}_\delta(\eta) = \overline{\mathsf{H}}_\delta\big(\overline{\eta}_\ast\big) = \overline{\Psi}_\delta\big(\lambda\vee\sqrt{\delta}\big)\geq (1+\sqrt{\delta})^2. 
\end{align*}
\end{lemma}
\begin{proof}
We first the monotonicity claim. For $\overline{\Psi}_\delta$, it is easy to calculate $\overline{\Psi}_\delta'(\lambda) = (\lambda^2-\delta)/\lambda^2\geq 0$ on $[\sqrt{\delta},\infty)$. This means $\overline{\Psi}_\delta$ is non-decreasing on $[\sqrt{\delta},\infty)$.

Next we shall solve the maximization problem. Let $\overline{\mathsf{G}}_\delta(\eta)\equiv \sqrt{1+\lambda \eta}+\sqrt{\delta(1-\eta)}$. Then we may compute 
\begin{align*}
\overline{\mathsf{G}}_\delta'(\eta)&= \frac{\lambda}{2\sqrt{1+\lambda \eta}}-\frac{\delta}{2\sqrt{\delta(1-\eta)}}.
\end{align*}
This means the map $\eta \mapsto \overline{\mathsf{H}}_\delta(\eta)=\overline{\mathsf{G}}_\delta^2(\eta)$ attains its maximum at $\overline{\eta}_\ast= \frac{(\lambda^2-\delta)_+}{\lambda(\delta+\lambda)}$, and with some calculations, we have
\begin{align*}
\max_{\eta \in [0,1]}\overline{\mathsf{H}}_\delta(\eta) = \overline{\mathsf{H}}_\delta\big(\overline{\eta}_\ast\big) = \overline{\Psi}_\delta\big(\lambda\vee\sqrt{\delta}\big)\geq \overline{\Psi}_\delta(\sqrt{\delta})=(1+\sqrt{\delta})^2, 
\end{align*}
as desired.
\end{proof}

\begin{lemma}\label{lem:sup_H}
	Let
	\begin{align*}
	\mathsf{H}_\delta(\eta)\equiv \Big(\sqrt{1+\lambda \eta}+\sqrt{\delta(1-\eta)}\Big)^2-(1+\lambda \eta),\quad \eta \in [0,1].
	\end{align*}
	Recall $\Psi_\delta$ defined in (\ref{def:psi_spiked_cov}), and let
	\begin{align}\label{def:eta_spiked_cov}
	\eta_\delta(\lambda)\equiv \frac{1}{2\lambda}\bigg[(\lambda-1)-\frac{\lambda+1}{\sqrt{\delta+4\lambda}}\cdot \sqrt{\delta}\bigg].
	\end{align}
	Then both $\Psi_\delta,\eta_\delta$ are non-decreasing on $[1+\sqrt{\delta},\infty)$ and
	\begin{align*}
	\max_{\eta \in [0,1]} \mathsf{H}_\delta(\eta) = \mathsf{H}_\delta(\eta_\ast) =\Psi_\delta\big(\lambda \vee \{1+\sqrt{\delta}\}\big)\geq 2\sqrt{\delta}+\delta
	\end{align*}
	with the maximizer $
	\eta_\ast \equiv \eta_\delta(\lambda)\cdot\big(\lambda \vee \{1+\sqrt{\delta}\}\big)$.	
\end{lemma}
\begin{proof}
	We first prove the monotonicity claim. For $\Psi_\delta$, patient calculations yield
	\begin{align*}
	\Psi_\delta'(\lambda) &= \frac{ \sqrt{\delta}}{2\lambda^2\sqrt{\delta+4\lambda}}\cdot\big(2\lambda^2-\sqrt{\delta}\sqrt{\delta+4\lambda}-2\lambda-\delta\big)\equiv \frac{ \sqrt{\delta}}{2\lambda^2\sqrt{\delta+4\lambda}} \cdot \psi_0(\lambda). 
	\end{align*}
	Note that $\psi_0(1+\sqrt{\delta})=0$ and for $\lambda\geq 1+\sqrt{\delta}$,
	\begin{align*}
	\psi_0'(\lambda) = 4\lambda -  \frac{2\sqrt{\delta}}{\sqrt{\delta+4\lambda}}-2>4\lambda-4\geq 0.
	\end{align*}
	This means that $\psi_0\geq 0$ on $I_\delta\equiv [1+\sqrt{\delta},\infty)$, and hence $\Psi_\delta' \geq 0$ on $I_\delta$, so $\Psi_\delta$ is non-decreasing on $I_\delta$. For $\eta_\delta$, again patient calculations lead to
	\begin{align*}
	\eta_\delta'(\lambda)= \frac{\delta^{3/2}+2\sqrt{\delta}\lambda^2+6\sqrt{\delta}\lambda+(\delta+4\lambda)^{3/2}
	}{2\lambda^2(\delta+4\lambda)^{3/2}}\geq 0,
	\end{align*}
	so $\eta_\delta$ is non-decreasing on $I_\delta$ as well. 
	
	Next we shall solve the maximum and the maximizer of $\eta \mapsto \mathsf{H}_\delta(\eta)$ over $\eta \in [0,1]$. Expanding the square in $\mathsf{H}_\delta$, we have
	\begin{align*}
	\mathsf{H}_\delta(\eta) = 2\sqrt{\delta}\cdot \sqrt{(1+\lambda \eta)(1-\eta)}+ \delta(1-\eta).
	\end{align*}
	So 	with $\pi_\delta(\eta)\equiv (1+\lambda \eta)(1-\eta)= \frac{(\lambda+1)^2}{4\lambda}-\lambda \big(\frac{\lambda-1}{2\lambda}-\eta\big)^2$,
	\begin{align*}
	\mathsf{H}'_\delta(\eta)&= \sqrt{\delta}\cdot \frac{ (\lambda-1)-2\lambda \eta}{\sqrt{(1+\lambda \eta)(1-\eta)}} - \delta = \sqrt{\delta}\cdot\left[\frac{2\lambda \big(\frac{\lambda-1}{2\lambda}-\eta\big)  }{\sqrt{\pi_\delta(\eta)}}-\sqrt{\delta}\right],\\
	\mathsf{H}''_\delta(\eta)&=-\frac{2\lambda \sqrt{\delta}}{\pi_\delta^{3/2}(\eta)}\bigg[\pi_\delta(\eta)+\lambda\bigg(\frac{\lambda-1}{2\lambda}-\eta\bigg)^2\bigg]\leq -\frac{4\lambda \sqrt{\lambda \delta}}{\lambda+1}.
	\end{align*}
	This means $\mathsf{H}'_\delta(0) = \sqrt{\delta}(\lambda-(1+\sqrt{\delta}))$ and $\mathsf{H}'_\delta$ is decreasing on $[0, (\lambda-1)_+/2\lambda]$. We now consider two cases:
	\begin{enumerate}
		\item First consider the case $\lambda \leq 1+\sqrt{\delta}$. Then $\mathsf{H}'_\delta(0)\leq 0$ so $\mathsf{H}'_\delta\leq 0$ globally on $[0,1]$. Therefore $\mathsf{H}_\delta$ is non-increasing on $[0,1]$ and the maximal value is
		\begin{align*}
		\max_{\eta \in [0,1]} \mathsf{H}_\delta(\eta) = \mathsf{H}_\delta(0) = 2\sqrt{\delta}+\delta. 
		\end{align*}
		\item Next consider the case $\lambda> 1+\sqrt{\delta}$. Then $\mathsf{H}'_\delta(0)>0$ and $\mathsf{H}'_\delta\big((\lambda-1)/2\lambda\big)<0$ which means that $\mathsf{H}'_\delta\geq 0$ on $[0, \eta_\ast]$ for some $\eta_\ast \in (0, (\lambda-1)/2\lambda)$ that solves $\mathsf{H}'_\delta(\eta_\ast)=0$. Some calculations show that
		\begin{align*}
		\eta_\ast = \eta_\delta(\lambda)=\frac{1}{2\lambda}\big[(\lambda-1)-\zeta_\ast\big]\geq 0,\quad \zeta_\ast \equiv \frac{\lambda+1}{\sqrt{\delta+4\lambda}}\cdot \sqrt{\delta},
		\end{align*}
		and the maximal value is
		\begin{align*}
		\max_{\eta \in [0,1]} \mathsf{H}_\delta(\eta) &= \mathsf{H}_\delta(\eta_\ast) = 2\zeta_\ast+\delta(1-\eta_\ast)\\
		& = \frac{\lambda+1}{\sqrt{\delta+4\lambda}}\bigg[2\sqrt{\delta}+\bigg(\frac{\sqrt{\delta}+\sqrt{\delta+4\lambda}}{2\lambda}\bigg)\cdot \delta\bigg]\\
		&= \Psi_\delta(\lambda)  \geq \Psi_\delta(1+\sqrt{\delta}) = 2\sqrt{\delta}+\delta. 
		\end{align*}
	\end{enumerate}
	The proof is complete.
\end{proof}

\begin{proof}[Proof of Proposition \ref{prop:spiked_cov}]
	In the proof we shall write $\Sigma_{r,\lambda}$ simply as $\Sigma$ for notational simplicity, and $h\sim \mathcal{N}(0,I_p)$. We shall perform detailed calculations for the case of sample covariance error using Lemma \ref{lem:sup_H}; the calculations for the other case are easier with the help of Lemma \ref{lem:sup_H_bar}, so the details will be omitted for simplicity.
	
	 Recall $\mathsf{F}_+$ defined in (\ref{def:F}), and $\Gamma_\lambda$ defined in (\ref{def:Gamma_lambda}). Then $\E \mathsf{F}_+(h)= E_\ast(\Sigma_{r,\lambda})$ and by Lemma \ref{lem:gw_ellipsoid},
	\begin{align*}
	\mathsf{F}_+(h) &= \sup_{ \alpha \in [0,\sqrt{1+\lambda}] } \Big\{\Big(\alpha+ \sqrt{\delta_n} \cdot \Gamma_\lambda (\alpha)+ \bigo(\Xi_n)\Big)^2-\alpha^2\Big\}.
	\end{align*}
	Here $\mathcal{O}$ is uniform in $\alpha \in [0,\sqrt{1+\lambda}]$.
	Now we may compute
	\begin{align*}
	& \bigabs{\E \mathsf{F}_+(h) - \sup_{\alpha \in [0,\sqrt{1+\lambda}]} \big\{\big(\alpha+ \sqrt{\delta}\cdot  \Gamma_\lambda(\alpha) \big)^2-\alpha^2\big\}}\\
	&\leq \E \sup_{\alpha \in [0,\sqrt{1+\lambda}]} \bigabs{\Big(\alpha+ \sqrt{\delta_n}\cdot  \Gamma_\lambda(\alpha)+\bigo(\Xi_n) \Big)^2-  \big(\alpha+ \sqrt{\delta}\cdot  \Gamma_\lambda(\alpha) \big)^2}  \\
	&\leq \E \sup_{\alpha \in [0,\sqrt{1+\lambda}]} \Big(\abs{\sqrt{\delta_n}-\sqrt{\delta}}\cdot \Gamma_\lambda(\alpha)+\bigo(\Xi_n) \Big) \Big(2\alpha+\abs{\sqrt{\delta_n}+\sqrt{\delta}}\cdot \Gamma_\lambda(\alpha)+\bigo(\Xi_n)\Big)\\
	&\lesssim \E^{1/2} \big(\abs{\sqrt{\delta_n}-\sqrt{\delta}}+\Xi_n\big)^2\cdot \E^{1/2}\bigg(\sqrt{\delta_n}+\sqrt{\delta}+ \sqrt{1+\lambda}\cdot \Big\{1\vee \frac{\pnorm{Z_{[1:r]}}{}}{\sqrt{n}}\Big\} \bigg)^2\\
	& \equiv A_1\cdot A_2. 
	\end{align*}
	For $A_1$,  we have
	\begin{align*}
	A_1^2&\lesssim \E \big(\sqrt{\delta_n}-\sqrt{\delta}\big)^2+ \frac{(1+\lambda)r}{n}\\
	&\lesssim \frac{1}{n}\cdot \Big[\var\big(\pnorm{Z_{(r:p]}}{}\big)+ \big(\E\pnorm{Z_{(r:p]} }{}-\sqrt{p-r}\big)^2+(1+\lambda)r\Big] \lesssim \frac{(1+\lambda)r}{n}. 
	\end{align*}
	For $A_2$, we have an easy estimate 
	\begin{align*}
	A_2^2 \lesssim \delta + (1+\lambda)\bigg(1\vee \frac{r}{n}\bigg) \asymp \delta+1+\lambda+\frac{(1+\lambda)r}{n}.
	\end{align*}
	Combining these estimates, we arrive at
	\begin{align}\label{ineq:spike_cov_4}
	& \bigabs{\E \mathsf{F}_+(h) - \sup_{\alpha \in [0,\sqrt{1+\lambda}]} \big\{\big(\alpha+ \sqrt{\delta}\cdot  \Gamma_\lambda(\alpha) \big)^2-\alpha^2\big\}}\nonumber\\
	&\lesssim (1+\lambda)\sqrt{\frac{r}{n}}\sqrt{1+\frac{\delta}{1+\lambda}+\frac{r}{n}}.
	\end{align}
	Finally we compute by using Lemma \ref{lem:sup_H} that
	\begin{align}\label{ineq:spike_cov_5}
	&\sup_{\alpha \in [0,\sqrt{1+\lambda}]} \big\{\big(\alpha+ \sqrt{\delta}\cdot  \Gamma_\lambda(\alpha) \big)^2-\alpha^2\big\}\nonumber\\
	& = \big\{(1+\sqrt{\delta})^2-1\big\} \vee \sup_{\alpha \in [1,\sqrt{1+\lambda}]}  \bigg\{\bigg(\alpha+ \sqrt{\delta}\cdot  \sqrt{1-\frac{\alpha^2-1}{\lambda}} \bigg)^2-\alpha^2\bigg\} \nonumber\\
	&=\big\{2\sqrt{\delta}+\delta\big\} \vee \sup_{\eta \in [0,1]} \bigg\{\Big(\sqrt{1+\lambda \eta}+\sqrt{\delta(1-\eta)}\Big)^2-(1+\lambda \eta)\bigg\} \nonumber\\
	& = \big\{2\sqrt{\delta}+\delta\big\} \vee \max_{\eta \in [0,1]} \mathsf{H}_\delta(\eta) = \Psi_\delta\big(\lambda \vee \{1+\sqrt{\delta}\}\big).
	\end{align}
	The claimed inequality follows by combining (\ref{ineq:spike_cov_4})-(\ref{ineq:spike_cov_5}).
\end{proof}

\appendix

\section{Proof of Theorem \ref{thm:CGMT}}\label{section:proof_CGMT}

We shall prove the following simpler version of Theorem \ref{thm:CGMT} with $n_1=n,n_2=0,m_1=m,m_2=0$. The proof of the case presented in Theorem \ref{thm:CGMT} differs by notational formalities.

We need Gordon's Gaussian min-max theorem \cite{gordon1985some,gordon1988milman}, formally recorded below:

\begin{theorem}[Gordon's Gaussian min-max theorem]\label{thm:gordon_min_max}
	Let $(X_{ij}),(Y_{ij})$, $i \in [n],j\in [m]$ be two centered Gaussian vectors such that
	\begin{enumerate}
		\item[(G1)] $\E X_{ij}^2 = \E Y_{ij}^2$ for all  $i,j \in [n]\times [m]$,
		\item[(G2)] $\E X_{ij}X_{ij'} \geq \E Y_{ij}Y_{ij'}$ for all $i \in [n], j,j' \in [m]$,
		\item[(G3)] $\E X_{ij} X_{i' j'}\leq \E Y_{ij} Y_{i' j'}$ for all $i\neq i' \in [n], j,j' \in [m]$.  
	\end{enumerate}
	Then for any $\{\lambda_{ij}: i\in[n], j\in [m]\}$, 
	\begin{align*}
	\Prob\bigg(\bigcap_{i \in [n]}\bigcup_{j \in [m]}\big\{X_{ij}>\lambda_{ij}\big\}\bigg)\leq \Prob\bigg(\bigcap_{i \in [n]}\bigcup_{j \in [m]}\big\{Y_{ij}>\lambda_{ij}\big\}\bigg).
	\end{align*}
\end{theorem}
A proof of this comparison inequality can be found in \cite[Corollary 3.13]{ledoux2013probability}.

\begin{proof}[Proof of Theorem \ref{thm:CGMT}]
	\noindent (1). Let $X(u,v)$ (resp. $Y(u,v)$) be the Gaussian process defined via
	\begin{align*}
	X(u,v)&= \pnorm{v}{} \iprod{g}{u} + \pnorm{u}{} \iprod{h}{v},\quad Y(u,v)=  \iprod{u}{G v} + \pnorm{u}{}\pnorm{v}{}z
	\end{align*}
	where $z$ is an independent standard normal random variable. Now we may compute the covariance structure:
	\begin{align*}
	\E X(u,v)X(u',v') 
	& = \iprod{u}{u'}\cdot \pnorm{v}{}\pnorm{v'}{}+ \iprod{v}{v'}\cdot \pnorm{u}{}\pnorm{u'}{},\\
	\E Y(u,v)Y(u',v') 
	& = \iprod{u}{u'}\cdot \iprod{v}{v'} + \pnorm{u}{}\pnorm{u'}{}\cdot  \pnorm{v}{}\pnorm{v'}{}.
	\end{align*}
	So we have
	\begin{align*}
	&\E Y(u,v)Y(u',v') - \E X(u,v)X(u',v') \\
	&= \big(\pnorm{u}{}\pnorm{u'}{}-\iprod{u}{u'}\big)\big(\pnorm{v}{}\pnorm{v'}{}-\iprod{v}{v'}\big)\geq 0. 
	\end{align*}
	This verifies (G1), (G3) in Theorem \ref{thm:gordon_min_max}, where (G2) is verified with equality. Consequently, the left tails satisfy
	\begin{align*}
	&\Prob\bigg(\min_{u \in U}\max_{v \in V} \big(Y(u,v)+Q(u,v)\big) \leq t \bigg)\\
	&\leq \Prob\bigg(\min_{u \in U}\max_{v \in V} \big(X(u,v)+Q(u,v)\big) \leq t \bigg)\leq \Prob\big(\Phi^{\textrm{a}}(g,h)\leq t\big). 
	\end{align*}
	So we only need to `eliminate' the `twist term' $\pnorm{u}{}\pnorm{v}{}z$ in $Y(u,v)$. This is done by a simple conditioning
	\begin{align*}
	&\Prob\bigg(\min_{u \in U}\max_{v \in V} \big(Y(u,v)+Q(u,v)\big) \leq t \bigg) \\
	&\geq \frac{1}{2} \Prob\bigg(\min_{u \in U}\max_{v \in V} \big(Y(u,v)+Q(u,v)\big) \leq t \Big|z\leq 0\bigg)\\
	& \geq \frac{1}{2} \Prob\bigg(\min_{u \in U}\max_{v \in V} \big(\iprod{u}{G v}+Q(u,v)\big) \leq t  \bigg) = 
	\frac{1}{2}\cdot \Prob\big( \Phi^{\textrm{p}}_+(G)\leq t\big).
	\end{align*}
	The cost for eliminating this twist term is a boosted constant factor $2$ in the left tail estimate. 
	
	\noindent (2). By flipping the role of $u,v$, replacing $(Q,t)$ with $(-Q,-t)$, and using $(G,g,h)\equald (-G,-g,-h)$ we have
	\begin{align*}
	\Prob\big( \Phi^{\textrm{p}}_-(G)\geq t\big)&= \Prob\bigg(\min_{v \in V}\max_{u \in U} \Big( -\iprod{u}{G v} - Q(u,v) \Big) \leq -t\bigg)\\
	&\leq 2\Prob\bigg(\min_{v \in V}\max_{u \in U} \Big(-\pnorm{v}{} \iprod{g}{u} - \pnorm{u}{} \iprod{h}{v}- Q(u,v)\Big)\leq -t\bigg)\\
	&\stackrel{(\ast)}{\leq} 2\Prob\bigg(\max_{u \in U}\min_{v \in V} \Big(-\pnorm{v}{} \iprod{g}{u} - \pnorm{u}{} \iprod{h}{v}- Q(u,v)\Big)\leq -t\bigg)\\
	& = 2 \Prob\big(\Phi^{\textrm{a}}(g,h)\geq t\big).
	\end{align*}
	Here $(\ast)$ follows by using $\max \min \leq \min \max$. 
\end{proof}

\section*{Acknowledgments}
The author would like to sincerely thank Ramon van Handel for several helpful correspondence to an earlier version of the paper.

\bibliographystyle{amsalpha}
\bibliography{mybib}

\providecommand{\bysame}{\leavevmode\hbox to3em{\hrulefill}\thinspace}
\providecommand{\MR}{\relax\ifhmode\unskip\space\fi MR }
\providecommand{\MRhref}[2]{%
  \href{http://www.ams.org/mathscinet-getitem?mr=#1}{#2}
}
\providecommand{\href}[2]{#2}
\begin{thebibliography}{ALPTJ10}

\bibitem[AAGM15]{artstein2015asymptotic}
Shiri Artstein-Avidan, Apostolos Giannopoulos, and Vitali~D. Milman,
  \emph{Asymptotic geometric analysis. {P}art {I}}, Mathematical Surveys and
  Monographs, vol. 202, American Mathematical Society, Providence, RI, 2015.
  \MR{3331351}

\bibitem[ALPTJ10]{adamczak2010quantitative}
Rados{\l}aw Adamczak, Alexander~E. Litvak, Alain Pajor, and Nicole
  Tomczak-Jaegermann, \emph{Quantitative estimates of the convergence of the
  empirical covariance matrix in log-concave ensembles}, J. Amer. Math. Soc.
  \textbf{23} (2010), no.~2, 535--561. \MR{2601042}

\bibitem[BBAP05]{baik2005phase}
Jinho Baik, G\'{e}rard Ben~Arous, and Sandrine P\'{e}ch\'{e}, \emph{Phase
  transition of the largest eigenvalue for nonnull complex sample covariance
  matrices}, Ann. Probab. \textbf{33} (2005), no.~5, 1643--1697. \MR{2165575}

\bibitem[BBvH23]{bandeira2023matrix}
Afonso~S. Bandeira, March~T. Boedihardjo, and Ramon van Handel, \emph{Matrix
  concentration inequalities and free probability}, Invent. Math. \textbf{234}
  (2023), no.~1, 419--487. \MR{4635836}

\bibitem[BCSvH24]{bandeira2024matrix}
Afonso~S Bandeira, Giorgio Cipolloni, Dominik Schr{\"o}der, and Ramon van
  Handel, \emph{Matrix {C}oncentration {I}nequalities and {F}ree {P}robability
  {II}. {T}wo-sided {B}ounds and {A}pplications}, arXiv preprint
  arXiv:2406.11453 (2024).

\bibitem[BvH22]{brailovskaya2022universality}
Tatiana Brailovskaya and Ramon van Handel, \emph{Universality and sharp matrix
  concentration inequalities}, arXiv preprint arXiv:2201.05142 (2022).

\bibitem[BY93]{bai1993limit}
Z.~D. Bai and Y.~Q. Yin, \emph{Limit of the smallest eigenvalue of a
  large-dimensional sample covariance matrix}, Ann. Probab. \textbf{21} (1993),
  no.~3, 1275--1294. \MR{1235416}

\bibitem[CHZ22]{cai2022nonasymptotic}
T.~Tony Cai, Rungang Han, and Anru~R. Zhang, \emph{On the non-asymptotic
  concentration of heteroskedastic {W}ishart-type matrix}, Electron. J. Probab.
  \textbf{27} (2022), Paper No. 29, 40. \MR{4385832}

\bibitem[CMW23]{celentano2023lasso}
Michael Celentano, Andrea Montanari, and Yuting Wei, \emph{The {L}asso with
  general {G}aussian designs with applications to hypothesis testing}, Ann.
  Statist. \textbf{51} (2023), no.~5, 2194--2220. \MR{4678801}

\bibitem[CS23]{chen2023gaussian}
Wei-Kuo Chen and Arnab Sen, \emph{On {$\ell_p$}-{G}aussian-{G}rothendieck
  problem}, Int. Math. Res. Not. IMRN (2023), no.~3, 2344--2428. \MR{4565615}

\bibitem[Dir15]{dirksen2015tail}
Sjoerd Dirksen, \emph{Tail bounds via generic chaining}, Electron. J. Probab.
  \textbf{20} (2015), no. 53, 1--29. \MR{3354613}

\bibitem[DS01]{davidson2001local}
Kenneth~R. Davidson and Stanislaw~J. Szarek, \emph{Local operator theory,
  random matrices and {B}anach spaces}, Handbook of the geometry of {B}anach
  spaces, {V}ol. {I}, North-Holland, Amsterdam, 2001, pp.~317--366.
  \MR{1863696}

\bibitem[Dvo59]{dvoretzky1959theorem}
Aryeh Dvoretzky, \emph{A theorem on convex bodies and applications to {B}anach
  spaces}, Proc. Nat. Acad. Sci. U.S.A. \textbf{45} (1959), 223--226; erratum,
  1554. \MR{105652}

\bibitem[Dvo61]{dvoretzky1961some}
\bysame, \emph{Some results on convex bodies and {B}anach spaces}, Proc.
  {I}nternat. {S}ympos. {L}inear {S}paces ({J}erusalem, 1960), Jerusalem
  Academic Press, Jerusalem; Pergamon, Oxford, 1961, pp.~123--160. \MR{0139079}

\bibitem[GN16]{gine2015mathematical}
Evarist Gin\'{e} and Richard Nickl, \emph{Mathematical foundations of
  infinite-dimensional statistical models}, Cambridge Series in Statistical and
  Probabilistic Mathematics, [40], Cambridge University Press, New York, 2016.
  \MR{3588285}

\bibitem[Gor85]{gordon1985some}
Yehoram Gordon, \emph{Some inequalities for {G}aussian processes and
  applications}, Israel J. Math. \textbf{50} (1985), no.~4, 265--289.
  \MR{800188}

\bibitem[Gor88]{gordon1988milman}
\bysame, \emph{On {M}ilman's inequality and random subspaces which escape
  through a mesh in {${\bf R}^n$}}, Geometric aspects of functional analysis
  (1986/87), Lecture Notes in Math., vol. 1317, Springer, Berlin, 1988,
  pp.~84--106. \MR{950977}

\bibitem[Han23]{han2023noisy}
Qiyang Han, \emph{Noisy linear inverse problems under convex constraints:
  {E}xact risk asymptotics in high dimensions}, Ann. Statist. \textbf{51}
  (2023), no.~4, 1611--1638. \MR{4658570}

\bibitem[HR22]{han2022gaussian}
Qiyang Han and Huachen Ren, \emph{Gaussian random projections of convex cones:
  approximate kinematic formulae and applications}, arXiv preprint
  arXiv:2212.05545 (2022).

\bibitem[HS23]{han2023universality}
Qiyang Han and Yandi Shen, \emph{Universality of regularized regression
  estimators in high dimensions}, Ann. Statist. \textbf{51} (2023), no.~4,
  1799--1823. \MR{4658577}

\bibitem[HX23]{han2023distribution}
Qiyang Han and Xiaocong Xu, \emph{The distribution of ridgeless least squares
  interpolators}, arXiv preprint arXiv:2307.02044 (2023).

\bibitem[Joh01]{johnstone2001distribution}
Iain~M. Johnstone, \emph{On the distribution of the largest eigenvalue in
  principal components analysis}, Ann. Statist. \textbf{29} (2001), no.~2,
  295--327. \MR{1863961}

\bibitem[KL17]{koltchinskii2017concentration}
Vladimir Koltchinskii and Karim Lounici, \emph{Concentration inequalities and
  moment bounds for sample covariance operators}, Bernoulli \textbf{23} (2017),
  no.~1, 110--133. \MR{3556768}

\bibitem[KM05]{klartag2005empirical}
B.~Klartag and S.~Mendelson, \emph{Empirical processes and random projections},
  J. Funct. Anal. \textbf{225} (2005), no.~1, 229--245. \MR{2149924}

\bibitem[LMPV17]{liaw2017simple}
Christopher Liaw, Abbas Mehrabian, Yaniv Plan, and Roman Vershynin, \emph{A
  simple tool for bounding the deviation of random matrices on geometric sets},
  Geometric aspects of functional analysis, Lecture Notes in Math., vol. 2169,
  Springer, Cham, 2017, pp.~277--299. \MR{3645128}

\bibitem[LT11]{ledoux2013probability}
Michel Ledoux and Michel Talagrand, \emph{Probability in {B}anach {S}paces},
  Classics in Mathematics, Springer-Verlag, Berlin, 2011, Isoperimetry and
  processes, Reprint of the 1991 edition. \MR{2814399}

\bibitem[Men10]{mendelson2010empirical}
Shahar Mendelson, \emph{Empirical processes with a bounded {$\psi_1$}
  diameter}, Geom. Funct. Anal. \textbf{20} (2010), no.~4, 988--1027.
  \MR{2729283}

\bibitem[Mil71]{milman1971new}
Vitali~D. Milman, \emph{A new proof of the theorem of {A}. {D}voretzky on
  sections of convex bodies}, Funct. Anal. Appl \textbf{5} (1971), 28--37.

\bibitem[Min17]{minsker2017some}
Stanislav Minsker, \emph{On some extensions of {B}ernstein's inequality for
  self-adjoint operators}, Statist. Probab. Lett. \textbf{127} (2017),
  111--119. \MR{3648301}

\bibitem[MM21]{miolane2021distribution}
L\'{e}o Miolane and Andrea Montanari, \emph{The distribution of the {L}asso:
  uniform control over sparse balls and adaptive parameter tuning}, Ann.
  Statist. \textbf{49} (2021), no.~4, 2313--2335. \MR{4319252}

\bibitem[MP12]{mendelson2012generic}
Shahar Mendelson and Grigoris Paouris, \emph{On generic chaining and the
  smallest singular value of random matrices with heavy tails}, J. Funct. Anal.
  \textbf{262} (2012), no.~9, 3775--3811. \MR{2899978}

\bibitem[MPTJ07]{mendelson2007reconstruction}
Shahar Mendelson, Alain Pajor, and Nicole Tomczak-Jaegermann,
  \emph{Reconstruction and subgaussian operators in asymptotic geometric
  analysis}, Geom. Funct. Anal. \textbf{17} (2007), no.~4, 1248--1282.
  \MR{2373017}

\bibitem[MRSY23]{montanari2023generalization}
Andrea Montanari, Feng Ruan, Youngtak Sohn, and Jun Yan, \emph{The
  generalization error of max-margin linear classifiers: Benign overfitting and
  high-dimensional asymptotics in the overparametrized regime}, arXiv preprint
  arXiv:1911.01544v3 (2023).

\bibitem[Sto13]{stojnic2013framework}
Mihailo Stojnic, \emph{A framework to characterize performance of lasso
  algorithms}, arXiv preprint arXiv:1303.7291 (2013).

\bibitem[TAH18]{thrampoulidis2018precise}
Christos Thrampoulidis, Ehsan Abbasi, and Babak Hassibi, \emph{Precise error
  analysis of regularized {$M$}-estimators in high dimensions}, IEEE Trans.
  Inform. Theory \textbf{64} (2018), no.~8, 5592--5628. \MR{3832326}

\bibitem[Tal14]{talagrand2014upper}
Michel Talagrand, \emph{Upper and lower bounds for stochastic processes},
  Ergebnisse der Mathematik und ihrer Grenzgebiete. 3. Folge. A Series of
  Modern Surveys in Mathematics, vol.~60, Springer, Heidelberg, 2014, Modern
  methods and classical problems. \MR{3184689}

\bibitem[Tik18]{tikhomirov2018sample}
Konstantin Tikhomirov, \emph{Sample covariance matrices of heavy-tailed
  distributions}, Int. Math. Res. Not. IMRN (2018), no.~20, 6254--6289.
  \MR{3872323}

\bibitem[Tro16]{tropp2016expected}
Joel~A. Tropp, \emph{The expected norm of a sum of independent random matrices:
  an elementary approach}, High dimensional probability {VII}, Progr. Probab.,
  vol.~71, Springer, [Cham], 2016, pp.~173--202. \MR{3565264}

\bibitem[Ver18]{vershynin2018high}
Roman Vershynin, \emph{High-dimensional probability: An introduction with
  applications in data science}, Cambridge Series in Statistical and
  Probabilistic Mathematics, vol.~47, Cambridge University Press, Cambridge,
  2018. \MR{3837109}

\bibitem[vH17]{van2017structured}
Ramon van Handel, \emph{Structured random matrices}, Convexity and
  concentration, IMA Vol. Math. Appl., vol. 161, Springer, New York, 2017,
  pp.~107--156. \MR{3837269}

\bibitem[Zhi24]{zhivotovskiy2024dimension}
Nikita Zhivotovskiy, \emph{Dimension-free bounds for sums of independent
  matrices and simple tensors via the variational principle}, Electron. J.
  Probab. \textbf{29} (2024), Paper No. 13, 28. \MR{4693860}

\end{thebibliography}

\end{document}